\documentclass[oneside,12pt]{amsart}
\makeatletter
\g@addto@macro{\endabstract}{\@setabstract}
\newcommand{\authorfootnotes}{\renewcommand\thefootnote{\@fnsymbol\c@footnote}}%
\makeatother
\usepackage[utf8]{inputenc}
\usepackage{amsmath}
\usepackage{amsthm}
\usepackage{geometry}
\usepackage{setspace}
\usepackage{enumitem}
\usepackage{amsmath}
\usepackage{amsfonts}
\usepackage{amssymb}
\usepackage{latexsym}
\usepackage{amsthm}
\usepackage{verbatim}
\usepackage[pagewise]{lineno}
\usepackage[english]{babel}
\usepackage{fancyhdr}
\usepackage{fancyhdr}
\usepackage{graphicx}
\usepackage{lipsum}
\usepackage{textcomp}
\usepackage{gensymb}
\usepackage{amsfonts}
\usepackage{multirow}
\usepackage{subcaption}

\usepackage{setspace}
\usepackage{dutchcal}
\usepackage{enumitem}
\usepackage{amsmath}
\usepackage{amsfonts}
\usepackage{amssymb}
\usepackage{latexsym}
\usepackage{amsthm}
\usepackage{verbatim}
\usepackage{setspace}
\usepackage{verbatim}
\usepackage{longtable}
\usepackage{flexisym}
 \usepackage{breqn}
 \usepackage{longtable}
 \usepackage{amsmath, amssymb, etoolbox, expl3, mathtools, pgfkeys, pgfopts, xparse, xstring,tikz}
\usepackage[edge-length=1.5cm,root-radius=.2cm,label,ordering=Bourbaki,mark=o]{dynkin-diagrams}

\usepackage{microtype}  % %minutely improves kerning
\usepackage[intoc]{nomencl}
\usepackage{color}

\usepackage{keystroke}
\usepackage{xcolor}
\usepackage{listings}
\usepackage{setspace}
\usepackage{amsmath}
\usepackage{bm}
\usepackage{upgreek}
\usepackage{units}
\usepackage{lscape}
\usepackage{pbox}
\usepackage[tableposition=top]{caption}
\usepackage{float}
\floatstyle{plaintop}
\restylefloat{table}
\usepackage[version=3]{mhchem}
\usepackage{hyperref,bookmark,refcount}
\hypersetup{
    colorlinks=true, %set true if you want colored links
    linktoc=all,     %set to all if you want both sections and subsections linked
    linkcolor=blue,  %choose some color if you want links to stand out
}

\newtheorem{theorem*}{Theorem}
\newtheorem{corollary*}[theorem*]{Corollary}
\newtheorem{theorem}{Theorem}[section]
\newtheorem{corollary}[theorem]{Corollary}
\newtheorem{lemma}[theorem]{Lemma}
\newtheorem{proposition}[theorem]{Proposition}

\newtheorem{prop}[theorem]{Proposition}
\newtheorem{lem}[theorem]{Lemma}

\newtheorem{rem}[theorem]{Remark}

\newtheorem{dfn}[theorem]{Definition}

\usepackage[parfill]{parskip}

\usepackage{xltabular}
\usepackage{mathtools}

\usepackage[maxbibnames=99,
backend=biber,
style=numeric,
sorting=nyt,
url=false
]{biblatex}
\AtEveryBibitem{\clearfield{doi}}
\AtEveryBibitem{\clearfield{eprint}}
\AtEveryBibitem{\clearfield{isbn}}
\AtEveryBibitem{\clearfield{issn}}
\begin{document}
\begin{center}
  \LARGE 
  Multiplicity-free representations of the principal $A_1$-subgroup in a simple algebraic group 
  \\[2ex]       
{\normalfont \normalsize \textcolor{black}{We dedicate this paper to the memory of the esteemed mathematician, Gary Seitz, whose work and mentorship have a continuing impact on the field and on our lives.}}    \par \bigskip

  \normalsize
  \authorfootnotes
  Aluna Rizzoli\textsuperscript{1}, Donna Testerman\footnote{Both authors acknowledge the support of the Swiss National Science Foundation grant number $200020\textunderscore207730$. In addition, the authors thank Gunter Malle, Mikko Korhonen and an anonymous referee for their helpful remarks on an earlier version of the manuscript.}\textsuperscript{2},
  \par \bigskip

  \textsuperscript{1}EPFL\\ \texttt{aluna.rizzoli@epfl.ch} \par
  \textsuperscript{2}EPFL\\ \texttt{donna.testerman@epfl.ch}\par \bigskip
\end{center}

%    Remove any unused author tags.

%    author one information
\thanks{Both authors acknowledge the support of the 1. In addition, the authors thank Gunter Malle, Mikko Korhonen and an anonymous for their helpful remarks on an earlier version of the mansucript.}

\keywords{}

\date{}

\begin{abstract}
\noindent Let $G$ be a simple algebraic group defined over an algebraically closed field $\mathcal{k}$ of characteristic $p>0$. For $p\geq h$, the Coxeter number of $G$, any regular unipotent element of $G$ lies in an $A_1$-subgroup of $G$; there is a unique $G$-conjugacy class of such subgroups and any member of this class is a so-called ``principal $A_1$-subgroup of $G$''. Here we classify all irreducible $\mathcal{k}G$-modules whose restriction to a principal $A_1$-subgroup of $G$ has no repeated composition factors, extending the work of Liebeck-Seitz-Testerman which treated the same question when $\mathcal{k}$ is replaced by an algebraically closed field of characteristic zero. 
\end{abstract}
\pagestyle{plain}
\section{Introduction}
In this paper, we consider a question in the representation theory and subgroup structure of simple algebraic groups defined over an algebraically closed
field $\mathcal{k}$ of characteristic $p>0$. The main aim of our work is to generalise the results of  \cite{LSTA1},  \cite{LST}, and \cite{LST2}, where the authors consider so-called ``multiplicity-free
subgroups'' of simple algebraic groups defined over an
algebraically closed field $K$ of characteristic 0. More precisely, 
the authors consider triples $(X,Y,V)$ where $X$ and $Y$ are simple algebraic groups defined over $K$ with $X$  a closed subgroup of $Y$, and $V$ is an
irreducible $KY$-module such that the $KX$-module $V$, obtained by restricting the action of $Y$ to the subgroup $X$,  is a sum of non-isomorphic irreducible
$KX$-modules  (a so-called ``multiplicity-free'' $KX$-module). The above cited articles provide a complete classification of such triples when either $X$
has rank 1 and does not lie in a proper parabolic subgroup of $Y$, or $Y$ is a classical group with natural module $W$ and  $X$ is of type $A_\ell$ acting
irreducibly on $W$. Note that the case where  $X$ acts irreducibly (and hence multiplicity freely) on $V$ was settled in work of
Dynkin \cite{dynkin}, in characteristic zero, and in work of Seitz and Testerman in positive characteristic \cite{seitz_mem_class, Tes_mem_exc}. 
\begin{comment}Further examples of such ``branching rules'' describing properties of the action of certain reductive subgroups of a reductive algebraic group on irreducible modules for the group are manifold in the literature. As examples, we mention  the work of Koike-Terada in \cite{KT}, where they establish explicit formulae giving the action of the symplectic or orthogonal group $H$ naturally embedded in $G = {\rm GL}_n(K)$ on irreducible representations of $G$, and the work of Brundan and Kleshchev in  \cite{BK}, where they describe the restriction of an irreducible $\mathcal{k}{\rm GL}_n(\mathcal{k})$-module to the subgroup $H = {\rm GL}_{n-1}(\mathcal{k})$.
\end{comment}

The ultimate far-reaching aim of what we undertake in this paper would be to investigate the ``multiplicity-free'' triples $(X,Y,V)$ as in \cite{LSTA1},  \cite{LST}, and \cite{LST2}, described above, replacing the field $K$ by the field $\mathcal k$ of positive characteristic $p$, and considering composition factors rather than summands. The proofs in \cite{LST,LST2} use induction on the rank of the group $X$; the case where $X$ is simple of rank 1 is considered in \cite{LSTA1}. 
 Here we treat the rank one case for the groups defined over $\mathcal{k}$, but consider a slightly more general setting than would strictly speaking be required for use in an inductive set-up.  Namely, we consider all
 simple algebraic groups $G$ (classical and exceptional), defined over $\mathcal {k}$, and $A$ a closed $A_1$-subgroup of $G$ containing a regular unipotent element of $G$, which
 we will call
a ``principal
$A_1$-subgroup of $G$''. (Such subgroups exist precisely when $p\geq h$, the Coxeter number of $G$; see \cite[Cor.~0.5 and Theorem~0.1]{A1_testerman}. In addition, there is at most one conjugacy class of principal $A_1$-subgroups in $G$; see \cite[Theorem 1.1]{seitz_tilting}.) We then determine
all irreducible $\mathcal{k}G$-modules $V$ such that the set of  composition factors of the $\mathcal{k}A$-module $V$ consists of non-isomorphic $\mathcal{k}A$-modules, and obtain a classification analogous to
\cite[Thm.~1]{LSTA1}.  Much of the analysis follows the same line of reasoning as that used in \cite{LSTA1}; the main differences and difficulties arise from the lack of 
precise knowledge about the dimensions of irreducible $\mathcal{k}G$-modules and the multiplicities of their weights. In addition, while irreducible 
$\mathcal{k}A_1$-modules are completely understood, the description of the set of weights is not as simple as in characteristic zero. In \cite{LST,LST2}, another essential ingredient of the proof is the work of Stembridge \cite{Stembridge}, where he determines when the tensor product of two irreducible modules for a simple algebraic group defined over the field $K$ is a direct sum of non-isomorphic irreducible modules. There has been recent progress on the analogous question for the simple groups defined over fields of positive characteristic in \cite{Gruber} and \cite{Gruber_Mancini}. The combination of the rank-$1$ theorem proven here and the work of  Gruber and Gruber-Mancini lays the foundation for the study of multiplicity-free subgroups  of higher rank for groups defined over fields of positive characteristic.
\bigbreak

In order to state our main result, we introduce some notation; further notation will be set up in Section~\ref{notation section}.   Fix $G$ a simply connected simple algebraic
group of rank $\ell\geq 2$
defined over the algebraically closed field $\mathcal{k}$.  We fix a maximal torus $T$ of $G$, a Borel subgroup $B$ of $G$ with $T\subset B$,
the root system $\Phi$ of $G$ with respect to $T$, and a base $\Pi=\{\alpha_1,\dots,\alpha_\ell\}$ of $\Phi$, associated with
the choice of Borel subgroup $B$. Let $\Phi^+$ be the associated set of positive roots. Let $X(T)$ denote the associated weight lattice, with fundamental dominant weights  $\{\omega_1,\dots,\omega_\ell\}$
defined by the choice of $\Pi$. (We label Dynkin diagrams as in \cite{bourbaki46}.) Throughout, we fix $\lambda\in X(T)$ a dominant weight and 
set $V = L(\lambda)$,
the irreducible $\mathcal{k}G$-module with highest weight $\lambda$. Assume that $p\geq h$, so that each regular unipotent
element of $G$ lies in an $A_1$-subgroup of $G$.  Let $A\subset G$ be a principal $A_1$-subgroup of
$G$. Fix a maximal torus $T_A$ of $A$ with $T_A\subset T$ and $T_AU_\alpha$, a Borel subgroup of $A$, with root group $U_\alpha$, lying in $B$. For
$\alpha$ the unique positive root of $A$ (with respect to the given choices), we have $T_A = \alpha^\vee(\mathcal{k}^*)$, the image of the co-root $\alpha^\vee$.  Henceforth, we will write $V\downarrow H$ for the $kH$-module obtained by restricting the action of $G$ to a subgroup $H$.
We say that \emph{$V\downarrow A$ is MF}
if all composition factors in the restriction are non-isomorphic.

We will also require a notation for the corresponding modules and subgroups for the groups defined over the algebraically closed field $K$ of characteristic $0$. We write $G_K$ for a simply connected
simple algebraic group
defined over the field $K$, with root system of type $\Phi$, and $A_K$ for a principal $A_1$-subgroup of $G_K$ (see \cite{Jacobson1951, Morozov1942} for the proofs of existence and conjugacy of $A_1$-subgroups of $G_K$ intersecting the class of regular unipotent elements). For the weight
$\lambda$ as above, we write $\Delta_K(\lambda)$ for the corresponding irreducible $G_K$-module. We will use the same terminology of ``MF'' for the action of $A_K$ on $\Delta_K(\lambda)$. Our main result is

\begin{theorem*}\label{main theorem} Suppose that $\lambda$ is $p$-restricted. Then $L(\lambda)\downarrow A$ is $MF$ if and only if one of the following holds.\begin{enumerate}[label=\textnormal{(\roman*)}]
  \item We have that $p>(\lambda\downarrow T_A)$ and $\Delta_K(\lambda)\downarrow A_K$ is MF.
    \item The group $G$ is of type $A_2$, $\lambda = \omega_1+\omega_2$ and $p=3$.
      \item The group $G$ is of type $B_2$, $\lambda = 2\omega_1$ and $p=5$.
     \end{enumerate}

  \end{theorem*}

\begin{corollary*}\label{main corollary}
    Let $\lambda = \sum _{i=0}^t p^i \lambda_i$ where each $\lambda_i$ is a $p$-restricted dominant weight. Then $L(\lambda)\downarrow A$ is MF if and only if one of the following holds.
    \begin{enumerate}[label=\textnormal{(\roman*)}]
  \item The module $\Delta_K(\lambda_i)\downarrow A_K$ is MF and $p>(\lambda_i\downarrow A)$, for all $0\leq i\leq t$.
    \item The group $G$ is of type $A_2$, $p=3$ and there exists $0\leq i\leq t$ such that $\lambda_i = \omega_1+\omega_2$. Moreover, for all $0\leq j\leq t$ we have  $\lambda_j \in \{0,\omega_1+\omega_2,\omega_1,\omega_2\}$ and if $\lambda_j = \omega_1+\omega_2$ for some $0\leq j\leq t-1$, then $\lambda_{j+1} = 0$.
      \item The group $G$ is of type $B_2$, $p=5$ and there exists $0\leq i\leq t$ such that $\lambda_i = 2\omega_1$. Moreover, for all $0\leq j\leq t$ we have  $\lambda_j \in \{0,2\omega_1,\omega_1,\omega_2\}$ and if $\lambda_j = 2\omega_1$ for some $0\leq j\leq t-1$, then $\lambda_{j+1}\in \{0,\omega_2\} $.
     \end{enumerate}
    
\end{corollary*}

  For the reader's convenience and for completeness, we list in Table~\ref{tab:p=0 big table} below the non-zero weights $\lambda$ for which
  $\Delta_K(\lambda)\downarrow A_K$ is MF, as obtained in \cite{LSTA1}.\bigbreak

 We conclude the introduction with a few remarks about the proof. We first note that if $p>(\lambda\downarrow T_A)$, then one can show that the Weyl module with highest weight
$\lambda$ is an irreducible $\mathcal{k}G$-module (see \cite[Cor.2.7.6]{mikkoThesis}), and then the considerations of \cite{LSTA1} for the groups defined over $K$ yield the result (see Proposition~\ref{V MF p=0 p>r iff lemma}).
The arguments therefore focus on the cases where $p\leq (\lambda\downarrow T_A)$. 
Many aspects of the proof follow closely the arguments used in \cite{LSTA1}. In particular, we use the fact
  that all irreducible $\mathcal{k}A_1$-modules have multiplicity one weight spaces and therefore considering the set of $T_A$-weights and their multiplicities in $V$ can directly be used
  to detect multiplicities of composition factors of $V\downarrow A$. Moreover, there are certain dimension bounds which must be respected by an MF-module. Thus, many of our preliminary lemmas are inspired by the
  results in
  \cite[Section~2]{LSTA1}. In addition, we rely on a result from \cite{McHa} where the authors prove that certain tilting modules for
  $G$ have a filtration by tilting modules for a principal $A_1$-subgroup $A$ of $G$. Since reducible indecomposable tilting modules for groups of type $A_1$ necessarily have repeated composition factors, this result is quite useful for
  showing that many $\mathcal{k}G$-modules are not MF as $\mathcal{k}A$-modules (see Proposition~\ref{CD lemma}).

\begin{center}
\begin{xltabular}[h]{\textwidth}{l l }
\caption{Multiplicity-free restrictions in characteristic $0$} \label{tab:p=0 big table} \\
\hline \multicolumn{1}{l}{$G_K$} & \multicolumn{1}{l}{Weight $\lambda$}\\
\hline 
\endhead
$A_\ell$ & $\omega_1,\, \omega_2,\, 2\omega_1,\, \omega_1+\omega_\ell$\\
         & $\omega_3\, (5\leq\ell\leq 7)$\\
         & $3\omega_1\, (\ell\leq 5),\, 4\omega_1 \,(\ell\leq 3),\, 5\omega_1 \,(\ell\leq 3)$\\
$A_3$    & $110$\\
$A_2$    & $c1,\,c0$\\ \hline
$B_\ell$ & $\omega_1,\, \omega_2,\, 2\omega_1$\\
         & $\omega_\ell\, (\ell\leq 8)$\\
$B_3$    & $101,\,002,\,300$\\
$B_2$    & $b0,\,0b\,(1\leq b\leq 5),\, 11,\, 12,\, 21$\\ \hline
$C_\ell$ & $\omega_1,\, \omega_2,\, 2\omega_1$\\
         & $\omega_3\, (3\leq\ell\leq 5)$\\  
         & $\omega_\ell\, (\ell=4,\, 5)$\\
$C_3$    & $300$\\ \hline
$D_\ell \, (\ell\geq 4)$ & $\omega_1,\, \omega_2\, (\ell = 2m+1),\, 2\omega_1\, (\ell = 2m)$\\
         & $\omega_\ell\, (\ell\leq 9)$\\ \hline
$E_6$ & $\omega_1,\omega_2$ \\
$E_7$ & $\omega_1,\omega_7$ \\
$E_8$ & $\omega_8$ \\
$F_4$ & $\omega_1,\omega_4$ \\
$G_2$ & $10,\,01,\,11,\,20,\,02,\,30$ \\
 \hline
\end{xltabular}
\end{center}

\section{Preliminary lemmas}\label{notation section}
Let us fix the notation to be used throughout the paper, in addition to that which was introduced in the previous section.  

Recall that $G$ is a simply connected simple algebraic group with principal $A_1$-subgroup $A$. We assume throughout that $\ell\geq 2$, respectively $2,3,4$, for $G$ of type $A_\ell$, respectively $B_\ell, C_\ell, D_\ell$. For $1\leq i\leq \ell$, let $s_i$ denote the simple reflection associated to the root $\alpha_i$. For $\lambda\in X(T)$, a dominant weight, we write $\Delta_G(\lambda)$ for the Weyl module, and $L_G(\lambda)$ for the irreducible module, for $G$ of highest weight $\lambda$. We will suppress the $G$ in this notation if there is no ambiguity. For a $\mathcal{k}G$-module $V$ and $\mu\in X(T)$, we write $V_\mu$ for the $\mu$-weight space with respect to $T$ of the module $V$. When we say that roots are adjacent, or end-nodes, we mean with respect to the Dynkin diagram associated to the root system $\Phi$. 

For a group of type $A_1$, we identify the weight lattice of a fixed maximal torus with the ring $\mathbb Z$ and write $(s)$ for the irreducible $\mathcal{k}A_1$-module of highest weight $s$. If we want to underline that we are talking about the fixed principal $A_1$-subgroup $A$, we may write as well $L_A(s)$. Similarly, we write $\Delta(s)$ for the Weyl module of highest weight $s$ and $T(s)$ for the indecomposable tilting module of highest weight $s$.  For a $\mathcal{k}A$-module $(s)$, we write $(s)^{(p^i)}$ for the module whose structure is induced by the composition of the $p^i$-Frobenius map on $A$ and the morphism defining the module structure on $(s)$.  For $A_K$, the principal $A_1$-subgroup of $G_K$, and $s$ a non-negative integer, we will write $\Delta_{A_K}(s)$ for the irreducible $KA_K$-module of highest weight $s$.

Here and in Sections~\ref{rank 2 section} and \ref{rank at least 3 section}, we fix a $p$-restricted dominant weight $\lambda\in X(T)$ and set $V = L_G(\lambda)$. Throughout the paper, set $r = \lambda\downarrow T_A$, that is $\lambda(\alpha^\vee(c)) = c^r$, for all $c\in \mathcal{k}^*$. Note that the cocharacter $\alpha^\vee : \mathbb G_m\to T$ which defines the maximal torus of $A$ satisfies $\alpha_i(\alpha^\vee(c)) = c^2$ for all $c\in \mathcal{k}^*$; that is, $\alpha_i\downarrow T_A = 2$ for all $1\leq i \leq \ell$. The value for $r$ can then be determined by writing $\lambda$ as a linear combination of simple roots and then using that each simple root takes value $2$ on $T_A$. We list the values of $r$ in Table~\ref{tab:r values}. 

\begin{longtable}{l l}
\caption{Values of $r=\lambda\downarrow T_A$ for $\lambda = \sum_1^\ell{c_i}\omega_i$} \label{tab:r values} \\

\hline \multicolumn{1}{l}{\textbf{$G$}} & \multicolumn{1}{l}{{$r$}}  \\ \hline 
\endhead
$A_\ell$ & $\sum_1^\ell i(\ell+1-i)c_i$ \\
$B_\ell$ & $\sum_1^{\ell-1} i(2\ell+1-i)c_i +\frac{\ell(\ell+1)}{2}c_\ell$ \\
$C_\ell$ & $\sum_1^\ell i(2\ell-i)c_i$ \\
$D_\ell$ & $\sum_1^{\ell-2} i(2\ell-1-i)c_i +\frac{\ell(\ell-1)}{2}c_{\ell-1}+\frac{\ell(\ell-1)}{2}c_{\ell}$ \\ \hline
$G_2$ & $6c_1+10c_2$ \\
$F_4$ & $22c_1+42c_2+30c_3+16c_4$ \\
$E_6$ & $16c_1+22c_2+30c_3+42c_4+30c_5+16c_6$ \\
$E_7$ & $34c_1 +49c_2 +66c_3 +96c_4 +75c_5 +52c_6 +27c_7$ \\ 
$E_8$ & $92c_1 + 136c_2 + 182c_3 + 270c_4 + 220c_5 + 168c_6 + 114c_7 + 58c_8$ \\\hline
\end{longtable}

In addition, recall that the existence of a principal $A_1$-subgroup  in $G$ implies that $p\geq h$, the Coxeter number of $G$, a hypothesis which allows us to apply the following proposition, a consequence of \cite[Theorem 1]{premet}

\begin{prop}\label{prop:premet} Let $p\geq h$ and let $\mu$ be a $p$-restricted weight for $G$. Then  the irreducible $\mathcal{k}G$-module $L(\mu)$ has precisely the same set of weights as the $\mathcal{k}G$-module $\Delta(\mu)$.
    
\end{prop}

\begin{proof} This follows from \cite[Theorem 1]{premet} since the parameter $e(\Phi)$ appearing in the statement of \emph{loc.cit.} is the maximum of the squares of the ratios of the lengths of the roots in $\Phi$.\end{proof}
We now introduce a shorthand notation for weights of $V$.
For $\lambda-\sum_{i=1}^\ell a_i\alpha_i$, we write $\lambda-i_1^{a_{i_1}}\cdots i_m^{a_{i_m}}$, where $a_j=0$ for $j\not\in\{i_1,\dots i_m\}$, and suppress those $a_j$ with $a_j=1$; for example, the weight $\lambda-\alpha_2-2\alpha_3-\alpha_5$ will be written as $\lambda-23^25$. For $G$ of rank 2, we write $\lambda-ab$ for the weight $\lambda-a\alpha_1-b\alpha_2$. 

The following result is Corollary 2.7.6 from  \cite{mikkoThesis}; we include a sketch of the proof for completeness.

\begin{lemma}\label{p>r irreducible Weyl module lemma}
    If $p>r$, then $\Delta(\lambda)$ is irreducible.
\end{lemma}

\begin{proof}
    By the Jantzen sum formula \cite[Part II, 8.19]{Jantzen}, it suffices to prove that for all $\alpha\in\Phi^+$, we have $r \geq \langle \lambda+\delta,\alpha\rangle -1$, where $\delta = \sum_{i=1}^\ell \omega_i$. It is easy to see that $\langle \lambda,\alpha \rangle$ is maximal when $\alpha$ is the highest root of the dual root system $\Phi^\vee$, i.e. when $\alpha$ is the highest short root $\beta$ of $\Phi$. By \cite[Prop.~5]{Serre1994}, we have $\langle \lambda+\delta,\beta\rangle \leq 1 +\sum_{\alpha\in\Phi^+}\langle \lambda ,\alpha\rangle$. It is therefore sufficient to show that $r =\sum_{\alpha\in\Phi^+}\langle \lambda ,\alpha\rangle $, which is a simple calculation using the fact that for a simple root $\alpha_i$ we have $\sum_{\alpha\in\Phi^+\setminus \{\alpha_i\}}\langle \alpha_i,\alpha\rangle = \sum_{\alpha\in\Phi^+\setminus \{\alpha_i\}}\langle s_i(\alpha_i),s_i(\alpha)\rangle = \sum_{\alpha\in\Phi^+\setminus \{\alpha_i\}}\langle -\alpha_i,\alpha\rangle$ and so $\sum_{\alpha\in\Phi^+\setminus \{\alpha_i\}}\langle \alpha_i,\alpha\rangle=0$. 
\end{proof}

The next proposition establishes Theorem~\ref{main theorem} when $p>r$.

\begin{proposition}\label{V MF p=0 p>r iff lemma}
    Assume $p>r$. Then $V\downarrow A$ is MF if and only if $\Delta_K(\lambda)\downarrow A_{K}$ is MF.
\end{proposition}
\begin{proof}
    By Lemma~\ref{p>r irreducible Weyl module lemma}, the Weyl module is irreducible and therefore $V = \Delta(\lambda)$. We have $\Delta_K(\lambda) \downarrow A_{K} = \sum_0^k \Delta_{A_{ K }}(r_i)$ for some integers $r_0\geq r_1\dots \geq r_k\geq 0$ with $r_0 = r$. Since $p>r$, a comparison of characters gives $\Delta(\lambda) \downarrow A = \sum_0^k \Delta_{A}(r_i) = \sum_{0}^k (r_i)$, which then implies the result.
\end{proof}

As the next result shows, in many cases when $\Delta(\lambda)$ is irreducible and $r\geq p$, we can still directly conclude that $V\downarrow A$ is not MF.

\begin{lem}\label{CD lemma}  Assume that $\Delta(\lambda)$ is irreducible, $r\geq p$, and $r\not\equiv -1\mod p$. Furthermore,  if $G$ is of type $B_\ell$, respectively $D_\ell$, and $\lambda$ does not lie in the root lattice of
  $G$, assume in addition that $p>{{\ell+1}\choose{2}}$, respectively $p> {\ell\choose2}$. Then $V\downarrow A$ is not MF.

\end{lem}

\begin{proof} Here we use \cite[Theorems 4.1.2, 4.1.4 and 4.2.1]{McHa} to see that $A$ is  a so-called ``good filtration'' subgroup, which then implies that the irreducible Weyl module $\Delta(\lambda) = V$
  has an  $A$-filtration by
  both Weyl modules and by induced modules. So in particular, $V\downarrow A$ is a tilting module. Furthermore, since $r$ is the highest $T_A$-weight in $V$, the module $T(r)$ is a summand of $V\downarrow A$. The hypotheses on $r$ imply that the indecomposable tilting module $T(r)$ is reducible (see \cite[Theorem 1.2]{Carter_Cline}). Since tilting modules for $A$ are self-dual, no reducible indecomposable
  tilting module is MF, which then concludes the proof.
  \end{proof}

We now turn to a sequence of definitions and lemmas which provide tools for studying the set of composition factors of $V\downarrow A$ based upon knowing the set of weights of $V$.

\begin{dfn} For $n\in\mathbb{Z}$, let $n_d$ be the multiplicity of the $T_A$-weight $r-2d$ in $V\downarrow A$ and let $m_d$ be the multiplicity of the composition factor $(r-2d)$ in $V\downarrow A$. Also, let $\mathbf{S}_d$ denote the multiset of composition factors whose highest weight is greater than $r-2d$ and in which $r-2d$ does not occur as a weight, and let $s_d$ denote the cardinality of $\mathbf{S}_d$.\end{dfn}

\begin{lemma}\label{multiplicity bound general lemma}
Assume that $V\downarrow A$ is MF. Then $n_d\leq d+1$.
\end{lemma}

\begin{proof}
Let $\mathbf{B}$ be the multiset of composition factors of $V\downarrow A$ where $r-2d$ occurs as a weight. Since $V\downarrow A$ is MF, we have $\mathbf{B}\subseteq \{(r),(r-2),\dots,(r-2d)\}$. Therefore $\lvert \mathbf{B} \rvert\leq d+1$ and we can conclude that $n_d\leq d+1$.
\end{proof}

\begin{lemma}\label{multiplicity recurrence lemma}
    For all $0\leq d\leq r$ we have \begin{equation}\label{recurrence equation}
        m_d = n_d-n_{d-1}+s_d-s_{d-1}.
    \end{equation}
    
\end{lemma}

\begin{proof}
    We prove this by induction on $d$. If $d=0$ the statement holds. Indeed, $m_0=n_0=1$ since $r$ is the highest weight and is afforded only by $\lambda$, and $n_{-1}=s_{-1}=s_0 = 0$.
    Assume that \eqref{recurrence equation} holds up to an arbitrary $d$. In general, the multiplicity $m_{d+1}$ can be determined by taking the difference between $n_{d+1}$ and the number of times the $T_A$-weight $r-2(d+1)$ appears in composition factors with greater highest weight. Thus,  \[ m_{d+1} = n_{d+1} -\left(\sum_{0\leq k\leq d}m_k-s_{d+1}\right).\]
    By the inductive hypothesis $m_k = n_k-n_{k-1}+s_k-s_{k-1}$ for all $k\leq d$. Substituting we get \[ m_{d+1} = n_{d+1} +s_{d+1} -\sum_{0\leq k\leq d}\left(n_k-n_{k-1}+s_k-s_{k-1}\right) = n_{d+1}-n_d+s_{d+1}-s_d,\] concluding the proof.
\end{proof}

\begin{lemma}\label{gaps containment up to r-2p}
    For all $1\leq d<p$ we have $\mathbf{S}_{d-1}\subseteq\mathbf{S}_d$. In particular $s_{d}\geq s_{d-1}$.
\end{lemma}
\begin{proof}
    We begin by analysing what weights occur in an arbitrary irreducible   module $(t)$. We will write $[ a_0,a_1,\dots,a_{m}] $ to denote the  integer $\sum_{i=0}^{m}a_i p^i$. Then $t = [a_0,a_1,\dots,a_{m}]$ where the $a_i$'s are the coefficients in the $p$-adic expansion of $t$. By Steinberg's 
    tensor product theorem, we have  \[(t) \cong (a_0)\otimes (a_1)^{(p)}\otimes\cdots\otimes (a_m)^{(p^m)}.\] The weights occurring in $(t)$ are therefore of the form $[a_0-2i_0,a_1-2i_1,\dots,a_{m}-2i_m] $ where $0\leq i_j \leq a_j$. Let $t-2q\geq 0$, with $q\in\mathbb{N}$, be an integer denoting a weight not occurring in $(t)$. Then $t-2q$ lies in an open interval \((\delta,\gamma)\) with 
    \begin{flalign*}
\delta &= [a_0,\dots, a_j, a_{j+1}-2i_{j+1}-2,\dots,a_{m}-2i_m], &\\
\gamma &= [-a_0,\dots, -a_j, a_{j+1}-2i_{j+1},\dots,a_{m}-2i_m], &
 \end{flalign*}
    
%\[\left((a_0,\dots, a_j, a_{j+1}-2i_{j+1}-2,\dots,a_{m}-2i_m), (-a_0,\dots, -a_j, a_{j+1}-2i_{j+1},\dots,a_{m}-2i_m)\right)\]
%\begin{equation}\label{gap interval equation}\left((a_0,\dots, a_j, a_{j+1}-2i_{j+1}-2,\dots,a_{m}-2i_m), (-a_0,\dots, -a_j, a_{j+1}-2i_{j+1},\dots,a_{m}-2i_m)\right) \end{equation}
where $0\leq i_{j+1}< a_{j+1}$ and $0\leq i_k\leq a_k$ for $k>j+1$. Conversely, any integer $t-2q$ lying in such an interval corresponds to a weight not occurring in $(t)$. We call these intervals the \textit{gaps} of $(t)$, so that a composition factor $(t)$ is in $ \mathbf{S}_d$ if and only if $r-2d$ is in a gap of $(t)$. 

Assume for a contradiction that $(t)\in \mathbf{S}_{d-1}\setminus \mathbf{S}_d$ for some $t\leq r$. Then $t>r-2d+2$ and the composition factor $(t)$ has a gap \((\delta,\gamma)\) as above containing $r-2d+2$, but not containing $r-2d$. This means that $r-2d = [a_0,\dots, a_j, a_{j+1}-2i_{j+1}-2,\dots,a_{m}-2i_m]$, implying that  \[2d - (r-t) = t-(r-2d) = t-[a_0,\dots, a_j, a_{j+1}-2i_{j+1}-2,\dots,a_{m}-2i_m]\] \[ = 2p^{j+1}[i_{j+1}+1,i_{j+2},\dots,i_m]\geq 2p.\] This contradicts the assumption that $d<p$.
\end{proof}

\begin{lemma}\label{sufficient conditions MF lemma}
Let $1\leq d<{\rm min}\{\lfloor \frac{r+2}{2}\rfloor,p\}$. Then the following hold.
\begin{enumerate}[label=\textnormal{(\roman*)}]
    \item If $n_d-n_{d-1} = 1$ then $r-2d$ is a composition factor of $V\downarrow A$.
    \item If $n_d-n_{d-1}\geq 2$ then $m_d\geq 2$ and $V\downarrow A$ is not MF.
    \item If $\lambda = c\omega_i$ and $n_d\geq d+1$ then $V\downarrow A$ is not MF.
    \item If $n_d-n_{d-1} = 1$ and $\mathbf{S}_{d-1}\neq \mathbf{S}_{d} $, then $m_d\geq 2$ and $V\downarrow A$ is not MF.
    
\end{enumerate}

\end{lemma}

\begin{proof}
    Parts (i),(ii) and (iv) follow directly from combining Lemma~\ref{multiplicity recurrence lemma} and Lemma~\ref{gaps containment up to r-2p}. If $\lambda = c\omega_i$ then $n_1 = 1$, and since $n_d\geq d+1$, there exists $2\leq d'\leq d$ such that $n_{d'}-n_{d'-1}\geq 2$, concluding by part (ii).
\end{proof}

We can often deduce the value $n_d$ from the characteristic zero case.

\begin{lemma}\label{n_d same as p=0 when Weyl module irreducible lemma}
    Assume that $V\cong\Delta(\lambda)$. Then $n_d = \dim(\Delta_K(\lambda)\downarrow A_{K})_{r-2d}$.
\end{lemma}
\begin{proof} This follows from \cite[Part II, 5.8]{Jantzen}, since $T_A$ is uniquely determined by the property $\alpha_i\downarrow\nolinebreak T_A =\nolinebreak 2$ for all $1\leq i\leq \ell$.\end{proof}

We now establish two dimension bounds for multiplicity-free $\mathcal{k}A_1$-modules.
\begin{dfn}\label{B(r)}
    Given $r\in\mathbb{N}$, define $B(r)$ and $B_{K}(r)$ as
    \[B(r)=\sum_{r-2k\geq 0} \dim L_A(r-2k),\]
    \[B_{K}(r)=\sum_{r-2k\geq 0} \dim \Delta_{A_{K}}(r-2k).\]
\end{dfn}

In particular note that $B_{K}(r)$ is either $(\frac{r}{2}+1)^2$  or $\frac{r+1}{2}\frac{r+3}{2}$ according to whether $r$ is even or odd, respectively.

\begin{lemma}\label{dimension bound comparison} We have $B(r)\leq B_{K}(r)$ and 
if $V\downarrow A$ is MF, then $\dim V\leq B(r)$.
\end{lemma}
\begin{proof}
The fact that $B(r)\leq B_{K}(r)$ is immediate since $\dim L_A(r-2k)\leq \dim \Delta_{A_{K}} (r-2k)$ for all $k$ such that $r-2k\geq 0$. Now if $V\downarrow A$ is MF, it can have at most one composition factor $(r-2d)$, i.e. $m_d = 1$, for every $0\leq d\leq \lfloor\frac{r}{2}\rfloor$. Therefore $\dim V\leq B(r)$. 
\end{proof}

\begin{lemma}\label{dimension bound for omega_i}
Suppose that $\lambda = a \omega_i$ and $r\not\equiv 0 \mod p$. If $V\downarrow A$ is MF, then $\dim V\leq B(r)-\dim (r-2)$.
\end{lemma}
\begin{proof}
The $T_A$-weight $r-2$ occurs with multiplicity $1$ in $V$, and since $r\not\equiv 0 \mod p$, it occurs as a weight in the composition factor $(r)$. Therefore $r-2$ does 
not afford a composition factor of $V\downarrow A$, i.e. $m_1 = 0$.
Since $V\downarrow A$ is MF, we have $m_d\leq 1$ for all $d\geq 0$ such that $r-2d\geq 0$. This proves that $\dim V\leq B(r)-\dim (r-2)$. 
\end{proof}

The following result is our main reduction tool, showing that if $V\downarrow A$ is MF, then $\lambda$ satisfies some highly restrictive conditions. The proof follows closely that of \cite[Lemma 2.6]{LST}.
\begin{prop}\label{first reduction lemma}
Let $\lambda = \sum_{i=1}^\ell c_i\omega_i$. Assume that there exist $i<j$ with $c_i\neq 0\neq c_j$ and that $V\downarrow A$ is MF. Then
\begin{enumerate}[label=\textnormal{(\roman*)}]
    \item $c_k=0$ for $k\neq i,j$.
    \item If $\alpha_i$ and $\alpha_j$ are non-adjacent, then $c_i=c_j=1$.
     \item If $\alpha_i$ and $\alpha_j$ are non-adjacent then they are both end-nodes.
    \item Either $\alpha_i$ or $\alpha_j$ is an end-node.
    \item If both $c_i>1$ and $c_j>1$, then $G$ has rank $2$ and $\lambda-ij$ has multiplicity $1$.
   \item If either $c_i>1$ or $c_j>1$, then either $G$ has rank $2$, or $\alpha_i$ is adjacent to $\alpha_j$ and $\lambda-ij$ has multiplicity $1$.
\end{enumerate}
\end{prop}

\begin{proof} We will use Proposition~\ref{prop:premet} throughout the proof, without direct reference.
\begin{enumerate}[label=\textnormal{(\roman*)}]
    \item If $c_k\geq 1$ for $k\neq i,j$, we have $n_1\geq 3$ as the $T_A$-weight $r-2$ is afforded by $\lambda-i$, $\lambda-j$ and $\lambda-k$. This contradicts Lemma~\ref{multiplicity bound general lemma}.
    \item Suppose $\alpha_i$ and $\alpha_j$ are not adjacent and that $c_i\geq 2$. Let $k\neq i,j$ such that $\alpha_{k}$ is adjacent to $\alpha_i$ and let $k'\neq i,j$ such that $\alpha_{k'}$ is adjacent to $\alpha_j$. Then $n_2\geq 4$, as the $T_A$-weight $r-4$ is afforded by $\lambda-i^2$, $\lambda-ik$, $\lambda-jk'$, $\lambda-ij$. This contradicts Lemma~\ref{multiplicity bound general lemma}.
     \item Assume that $\alpha_i$ and $\alpha_j$ are non-adjacent and that $\alpha_i$ is not an end-node. Then there exist distinct simple roots $\alpha_k, \alpha_l$, both adjacent to $\alpha_i$, and a simple root $\alpha_m\neq \alpha_i$ adjacent to $\alpha_j$. Then $n_2\geq 4$, as the $T_A$-weight $r-4$ is afforded by $\lambda-ik$, $\lambda-il$, $\lambda-jm$ and $\lambda-ij$. This contradicts Lemma~\ref{multiplicity bound general lemma}.
    \item Assume that neither $\alpha_i$ nor $\alpha_j$ is an end-node. Then by $(\rm{\Rn{3}})$, the roots $\alpha_i$ and $\alpha_j$ are adjacent. Let $1\leq k,l\leq \ell$ be distinct indices such that $\{i,j\}\cap\{k,l\}=\emptyset$ and such that $\alpha_i$ is adjacent to $\alpha_k$ and $\alpha_j$ is adjacent to $\alpha_l$. Then the $T_A$-weight $r-8$ is afforded by $\lambda -kijl, \lambda -kij^2,\lambda-i^2j^2, \lambda-i^2jl,\lambda-ki^2j$ and $\lambda-ij^2l$. Therefore $n_4\geq 6$, contradicting Lemma~\ref{multiplicity bound general lemma}. 
    \item If both $c_i>1$ and $c_j>1$, then by (ii), the roots $\alpha_i$ and $\alpha_j$ are adjacent. If the rank of $G$ is not $2$ we can find $k\neq i,j$, such that $\alpha_k$ is adjacent to either $\alpha_i$ or $\alpha_j$. But then the $T_A$-weight $r-4$ is afforded by $\lambda-ij$, $\lambda-i^2$, $\lambda-j^2$ and either $\lambda-ik$ or $\lambda-jk$. Therefore $n_2\geq 4$, contradicting  Lemma~\ref{multiplicity bound general lemma}. In addition, $\lambda-ij$ has multiplicity 1, else $n_2\geq 4$, again contradicting Lemma~\ref{multiplicity bound general lemma}.
   \item Assume $c_i\geq 2$ and that $G$ has rank at least $3$. Then $\alpha_i$ and $\alpha_j$ are adjacent by $({\rm{\Rn{2}}})$, and $r-4$ is afforded by $\lambda-i^2$, $\lambda-ij$ and either $\lambda-ik$ or $\lambda-jk$ for some $k\neq i,j$. Therefore, Lemma~\ref{multiplicity bound general lemma} implies that $\lambda-ij$ has multiplicity 1, as claimed.
\end{enumerate}
\end{proof}

\begin{lemma}\label{w_i reduction space on both sides}
Assume that $\lambda = \omega_i$ and that there exist $\{\beta_{i-3},\dots,\beta_{i+3}\}\subseteq \Pi$ such that for $i-3\leq s<t\leq i+3$, $(\beta_s,\beta_t) \ne 0$ if and only if $t = s+1$. Then $V\downarrow A$ is not MF.
\end{lemma}

\begin{proof}
Here $p>7$, as ${\rm rank}(G)\geq 7$, and Table~\ref{tab:r values} shows that $r>15$. It is now a simple check to see that $n_4\geq 5$, concluding by Lemma~\ref{sufficient conditions MF lemma}{(\Rn{3})}.
\end{proof}

\begin{lemma}\label{bw_i reduction lemma}
Assume that $\lambda = b\omega_i$ with $b\ge 2$. If $V\downarrow A$ is MF, then $\alpha_i$ is an end-node.
\end{lemma}

\begin{proof}
    If $\alpha_i$ is not an end-node, it is easy to see that $n_2\geq 3$. As ${\rm rank}(G)\geq 3$ we have $p>3$, and Table~\ref{tab:r values} shows that $r> 7$, so  Lemma~\ref{sufficient conditions MF lemma}{(iii)} implies that $V\downarrow A$ is not MF.
\end{proof}

\begin{rem} In the previous two proofs, we have applied Lemma~$\ref{sufficient conditions MF lemma}$, and in each case it was straightforward to see that the condition $d<{\rm min}\{\lfloor \frac{r+2}{2}\rfloor,p\}$ was satisfied. In what follows, we will apply the lemma without systematically pointing out how we conclude that the hypothesis holds.\end{rem}

The following lemma  provides a classification for the second possibility of Proposition~\ref{first reduction lemma}(vi).

\begin{lemma}\textup{\cite[\nopp1.35]{Tes_mem_exc}}\label{1 dimensional weight space lemma}
    Assume that $\lambda = c_i\omega_i+c_j\omega_j$ with $\alpha_i$ and $\alpha_j$ adjacent and $c_ic_j\ne 0$. Let $d =\dim V_{\lambda-ij}$. Then $1\leq d\leq 2$ and the following hold.
    \begin{enumerate}[label=\textnormal{(\roman*)}]
        \item If $(\alpha_i,\alpha_i)=(\alpha_j,\alpha_j)$, then $d=1$ if and only if $c_i+c_j = p-1$.
        \item If $(\alpha_i,\alpha_i)=2(\alpha_j,\alpha_j)$, then $d=1$ if and only if $2c_i+c_j +2 \equiv 0 \mod p$.
        \item If $(\alpha_i,\alpha_i)=3(\alpha_j,\alpha_j)$, then $d=1$ if and only if $3c_i+c_j +3 \equiv 0 \mod p$.
    \end{enumerate}
\end{lemma}

Finally, we conclude this section with two further results on the dimensions of certain weight spaces in $V$. 

\begin{lemma}\textup{\cite[\nopp 8.6]{seitz_mem_class}}\label{lambda minus string dimension lemma}
    Let $G=A_\ell$. Suppose that $\lambda = c_i\omega_i+c_j\omega_j$ and $1\leq s\leq i < j \leq t\leq k$, with $c_ic_j\neq 0$.  Let $d =\dim V_{\lambda-s(s+1)\dots (t-1)t}$. Then
        \begin{enumerate}[label=\textnormal{(\roman*)}]
        \item If $a+b+j-i\not\equiv 0 \mod p$, then $d = j-i+1$;
        \item If $a+b+j-i\equiv 0 \mod p$, then $d = j-i$.
    \end{enumerate}
\end{lemma}

\begin{lemma}\textup{\cite[Lemma~2.2.8]{IrrGeomSubCl_BGT}}\label{subdiagram lemma}
Let $\lambda = \sum_{i=1}^\ell d_i\omega_i$ and let $\mu = \lambda - \sum_{\beta\in S}c_{\beta}\beta\in X(T)$ for some subset $S\subseteq \Pi$. Set $X =\left\langle U_{\pm \beta}  \mid \beta\in S\right\rangle$, where for $\gamma\in\Phi$, $U_\gamma$ is the $T$-root subgroup associated to $\gamma$, $\lambda' = \lambda\downarrow (T\cap X)$ and $\mu' = \mu\downarrow (T\cap X)$. Then $\dim V_\mu = V'_{\mu'}$, where $V' = L_X(\lambda')$.
\end{lemma}

\section{The case where $G$ has rank $2$}\label{rank 2 section}
In this section we establish Theorem~\ref{main theorem} in the case where $G$ has rank $2$. Let us recall our setup. Throughout the section we assume that $\lambda$ is a $p$-restricted dominant weight for $G$, and we let $r = \lambda\downarrow T_A$ and $V = L(\lambda)$. We write $\lambda = ab$ as shorthand notation for $\lambda = a\omega_1+b\omega_2$, and $\lambda-ab$ for $\lambda-a\alpha_1-b\alpha_2$.
Furthermore we assume that $p\leq r$, as Proposition~\ref{V MF p=0 p>r iff lemma} settles the case $r<p$, and in addition, we have that $p\geq h$ since we are assuming the existence of a principal $A_1$-subgroup in $G$.

\subsection{The case where $G$ is $A_2$}
We begin with the case $G=A_2$, where $p\geq h = 3$. The following is the main result, which we will prove after a sequence of lemmas.
\begin{proposition}\label{a2 classification prop}
    Let $G = A_2$ and assume $p\leq r$. Then $V\downarrow A$ is MF if and only if $\lambda = \omega_1+\omega_2$ and $p = 3$.
\end{proposition}

\begin{lemma}\label{A2 weights}
Let $\lambda = ab$. Then $\lambda-ij$, with $i+j\leq a+b$, is a weight of $\Delta(\lambda)$ if and only if one of the following holds.
\begin{enumerate}[label=\textnormal{(\roman*)}]
    \item $i\leq j$ and $j-i\leq b$.
    \item $i\geq j$ and $i-j\leq a$.
\end{enumerate}
\end{lemma}

\begin{proof}
    By \cite[VIII, \S7, Prop. 10]{bourbaki78}, the weights in $\Delta(\lambda)$ are precisely the same as those occurring in $\Delta(a0)\otimes \Delta(0b)$. Let $\lambda_1 = a0$ and $\lambda_2 = 0b$ and recall that $\Delta(c_i\omega_i)$, for $i=1,2$ is the $c_i$-th symmetric power of the natural, respectively, dual module for $G$. Hence, $\lambda_1-i_1j_1$  is a weight of $\Delta(\lambda_1)$ if and only if $i_1+j_1\leq 2a$ and  $0\leq i_1-j_1\leq a$. Similarly, $\lambda_2-i_2j_2$  is a weight of $\Delta(\lambda_2)$ if and only if $i_2+j_2\leq 2b$ and  $0\leq j_2-i_2\leq b$. By symmetry it suffices to show that the statement of the lemma is valid when $i\geq j$. First of all, it is clear that all weights $\lambda-ij$ of $\Delta(\lambda)$ satisfy $i-j\leq a$, since a weight $\lambda_2-i_2j_2$ of $\Delta(\lambda_2)$  satisfies $i_2-j_2\leq 0$, and a weight $\lambda_1-i_1j_1$ of $\Delta(\lambda_1)$ satisfies $i_1-j_1\leq a$. 
    
    For the converse, consider a pair $(i,j)$, such that $i\geq j$, $i+j=d\leq a+b$ and $i-j\leq a$. If $d\leq a$, then $j\leq i\leq a$, so $\lambda_1-ij$ is a weight of $\Delta(\lambda_1)$ and $\lambda_1+\lambda_2-ij$ is then a weight of $\Delta(\lambda)$. If $d>a$, write $d=a+k$, where $k\leq b$. To conclude we show that we can find $(i_1,j_1)$ with $i_1+j_1 = a$ and  $(i_2,j_2)$ with $i_2+j_2 = k$, such that $\lambda_1-i_1j_1$ is a weight of $\Delta(\lambda_1)$, $\lambda_2-i_2j_2$ is a weight of $\Delta(\lambda_2)$ and $i_1-j_1+i_2-j_2 = i-j$. Fix $i_1+j_1=a$ and $i_2+j_2=k$. Note that we are allowed to pick $i_1-j_1$ to be any integer between $a$ and $1$ if $a$ is odd, and between $a$ and $0$ if $a$ is even. Similarly, we are allowed to choose $i_2-j_2$ between $-k$ and $-1$ if $k$ is odd, and between $-k$ and $0$ when $k$ is even, concluding easily. 
\end{proof}

\begin{lemma}\label{A2 weights patterns}
Let $\lambda = ab$ with $a\geq b>0$ and $a+b=p-1$. 
\begin{enumerate}[label=\textnormal{(\roman*)}]
    \item For $0\leq d \leq b$, we have $n_d=d+1$. 
    \item For $b+1\leq d \leq a$, we have that $n_d$ increases alternatingly by respectively $0$ and $1$ with respect to $n_{d-1}$.
    \item For $a<d\leq a+b$, we have that $n_d$ alternates between $\left\lceil\frac{a+b}{2}\right\rceil$ and $\left\lceil\frac{a+b+1}{2}\right\rceil$.
\end{enumerate}
\end{lemma}

\begin{proof}
Here we use the fact that all $T$-weights in $V$ are of multiplicity 1. (See \cite[Prop. 2]{Za_Su_1dim}. ) Hence, the proof consists of counting the pairs $(i,j)$ with $i+j=d$ and satisfying the conditions of Lemma~\ref{A2 weights}.
    \begin{enumerate}[label=\textnormal{(\roman*)}]
    \item Let $0\leq d\leq b$. The statement then follows immediately from noting that $\lambda - i(d-i)$ is a weight for $0\leq i\leq d$.
    \item Let us start from $d=b+1$, where the weights are given by $\lambda - (b-i+1)i$ for $0\leq i\leq b$. This means that $n_{b+1}=b+1=n_{b}$ by part (i). For $d=b+2$, still assuming that $d\leq a$, we find weights of the form $\lambda - (b-i+2)i$ for $0\leq i\leq b+1$. The same reasoning continues until $d=a$, proving the statement. 
    \item Let $a< d\leq a+b$. We must count the weights of the form $\lambda-i(d-i)$ where $a\geq 2i-d$ and $b\geq d-2i$. The conditions on $i$ are  equivalent to the inequalities $\frac{d-b}{2}\leq i\leq \frac{a+d}{2}$. Considering the various possibilities for the evenness of the terms in the inequality gives the result. 
\end{enumerate}\end{proof}

\begin{lemma}\label{A2 c1c2 special p lemma}
    Let  $\lambda = ab$ with $a\geq b>0$ and $a+b=p-1$. Then $V\downarrow A$ is MF if and only if $a=b=1$.
\end{lemma}
\begin{proof}
   Note that $r=2(a+b)<2p$ and that $a-b$ is an even number. For clarity we split the proof into $4$ cases, depending on whether $a-b\geq 6$, $a-b = 4$, $a-b=2$ or $a=b$. Suppose first that $a-b\geq 6$. By Lemma~\ref{A2 weights patterns}, all weights of the form $r-2d$ with $b+1\leq d \leq a$ follow the pattern in (ii) of the same lemma. Since \(r-2(b+1) = 2a-2\geq a+b+4 = p+3\), and \(r-2a = 2b \leq b+a-6 = p-7\), this includes weights that restrict to $p+3,p+1,p-1,p-3,p-5$. Therefore by Lemma~\ref{sufficient conditions MF lemma}{(i)} either $(p+3)$ or $(p+1)$ is a composition factor for $V\downarrow A$. In the first case $p-5$ occurs with multiplicity $1$ more than $p-3$, and does not occur as a weight in the composition factor $(p+3)$, while $p-3$ does. Therefore by Lemma~\ref{sufficient conditions MF lemma}{(iv)} the module $V\downarrow A$ is not MF. In the second case $(p-3)$ is similarly a repeated composition factor. 

    Now suppose that \(a-b=4\). We have \(p+3 = a+b+4 = r-2(\frac{a+b}{2}-2)=r-2b\). Therefore by Lemma~\ref{A2 weights patterns}(i), we have that \((p+3)\) is a composition factor for \(V\downarrow A\). The weights $p+1,p-1,p-3,p-5$ follow the pattern described in Lemma~\ref{sufficient conditions MF lemma}{(ii)}. Therefore we can conclude like in the previous case.
    
    Now suppose that $a=b+2$. Then by Lemma~\ref{A2 weights patterns} we know that $r-2k$ occurs with multiplicity $k+1$ for $k$ ranging between $0$ and $b$. In particular $(r-2b) = (p+1)$ is a composition factor by Lemma~\ref{sufficient conditions MF lemma}{(i)}. Again by Lemma~\ref{A2 weights patterns}, the $T_A$-weight $r-2(b+1) = 2b+2$ occurs with multiplicity $b+1$ and $r-2a=2b$ occurs with multiplicity $b+2$. Since $2b = p-5$ does not occur as a weight in the composition factor $(p+1)$, while $p-3 = 2b+2$ does, Lemma~\ref{sufficient conditions MF lemma}{(iv)} implies that $V\downarrow A$ is not MF. 
    
    Finally assume that $a=b$. Then by Lemma~\ref{A2 weights patterns} and Lemma~\ref{sufficient conditions MF lemma}{(i)}, the weights $r=4a,4a-2,\dots , 2a$ afford composition factors for $V\downarrow A$, with the last weight occurring with multiplicity $a+1$. If $a\geq 2$ we find that $2a-2$ occurs with multiplicity $\left\lceil\frac{a+b}{2}\right\rceil =a$ and $2a-4$ occurs with multiplicity $\left\lceil\frac{a+b+1}{2}\right\rceil=a+1$. Since $2a-4=p-5$ does not occur as a weight in the composition factor $(p+3)$, while $p-3$ does, Lemma~\ref{sufficient conditions MF lemma}{(iv)} implies that $V\downarrow A$ is not MF. On the other hand if $a=b=1$ we find that $V\downarrow A = (4)\oplus(2)$.
\end{proof}

\renewcommand*{\proofname}{Proof of Proposition~$\ref{a2 classification prop}$.}
\begin{proof}
Suppose that $V\downarrow A$ is MF, with $\lambda = ab$ and $a\geq b$. Since the Weyl module $\Delta(c0)$ is irreducible, the assumption that $r = 2a+2b\geq p>a$, together with Lemma~\ref{CD lemma}, implies that $b\geq 1$. If $\dim V_{\lambda-11} = 2$, then $a+b\neq p-1$ by Lemma~\ref{1 dimensional weight space lemma}, and $b = 1$ by Proposition~\ref{first reduction lemma}{(v)}. In this case, using the Jantzen $p$-sum formula \cite[Part II, 8.19]{Jantzen} (for example), one sees that  $\Delta(\lambda)$ is irreducible, a contradiction by Lemma~\ref{CD lemma}. If $\dim V_{\lambda-11} = 1$, then by Lemma~\ref{1 dimensional weight space lemma} we have $a+b= p-1$, and we conclude by Lemma~\ref{A2 c1c2 special p lemma}.
\end{proof}
\renewcommand*{\proofname}{Proof.}

\subsection{The case where $G$ is $B_2$}
We proceed with the case $G=B_2$, where $p\geq h = 4$. The main result is the following, which we shall prove after a series of lemmas.
\begin{proposition}\label{B2 classification prop}
    Let $G = B_2$ and assume that $p\leq r$. Then $V\downarrow A$ is MF if and only if $\lambda = 2\omega_1$ and $p = 5$.
\end{proposition}

We begin by recalling some information about the structure of $B_2$ Weyl modules (with $p$-restricted highest weights). Let $\lambda = ab$ be a $p$-restricted dominant weight; here $\alpha_1$ is long.

We consider the following alcoves in which a $p$-restricted weight can lie :

$C_0 = \{\lambda \ |\ 2a+b+3 <p\}$;

$C_1 = \{\lambda \ |\ a+b+2<p<2a+b+3\}$;

$C_2 = \{\lambda \ |\ b+1<p<a+b+2 \mbox{ and } 2a+b+3 <2p\}$;

$C_3 = \{\lambda \ |\ 2a+b+3 >2p \mbox{ and } {\rm max}\{b+1, a+1\}<p\}$.

\begin{lemma}\label{B2 Weyl} The following hold.
\begin{enumerate}[label=\textnormal{(\roman*)}]
\item If $\lambda \in C_i$ for $i=1,2,3$, then $\Delta(\lambda)$ has exactly two composition factors, namely $V$ and $L(\mu)$, where $\mu =
(p-a-b-3)\omega_1+b\omega_2$, respectively $a\omega_1+(2p-2a-b-4)\omega_2$, $(2p-a-b-3)\omega_1+b\omega_2$, when $i=1,2,3$.
\item For $\lambda = a\omega_1+(p-1)\omega_2$ with $2a+(p-1)+3>2p$ and $a<p-1$, we have that $\Delta(\lambda)$ has exactly two composition factors, $V$ and $L(\mu)$ for  $\mu = (p-a-2)\omega_1+(p-1)\omega_2$.
\end{enumerate}
For $\lambda$ a $p$-restricted dominant weight not lying in $\cup_{i=1}^3C_i$ and not of the form described in \normalfont{(ii)} above, $\Delta(\lambda)$ is irreducible.
\end{lemma}

\begin{proof} This follows from the Jantzen $p$-sum formula \cite[Part II, 8.19]{Jantzen}.
\end{proof}

\begin{rem}\label{B2 Weyl remark} Recall that here $\omega_1 = \alpha_1+\alpha_2$. It follows from Lemma~$\ref{B2 Weyl}$ that for a $p$-restricted weight $\lambda = ab$,  if $\Delta(\lambda)$ is reducible then the module $\Delta(\lambda)$ has exactly one composition factor in addition to the composition factor $L(\lambda)$. The highest weight of the second composition factor is of the form $(a-k)\omega_1+b\omega_2$ or $a\omega_1+(b-k)\omega_2$, for some $k\geq 1$. More precisely, for $\lambda\in C_i$, $i=1,2,3$ and $\mu$ as in the statement of the lemma, we have $\mu = \lambda-(2a+b+3-p)(\alpha_1+\alpha_2)$, respectively $\lambda-(a+b+2-p)(\alpha_1+2\alpha_2)$, $\lambda-(2a+b+3-2p)(\alpha_1+\alpha_2)$. And in case (ii) of the lemma, $\mu = \lambda-(2a-p+2)(\alpha_1+\alpha_2)$.
    
\end{rem}

 We record for convenience the dimension of the Weyl module $\Delta(ab)$, namely \[\dim \Delta(ab) = \frac{1}{6}(a+1)(b+1)(a+b+2)(2a+b+3).\]
 
\begin{lemma}\label{B2 weights c0}
Let $\lambda = c0$. Then $\lambda-ij$, with $i+j\leq 2c$, is a weight of $V$ if and only if one of the following holds.
\begin{enumerate}[label=\textnormal{(\roman*)}]
    \item $i\leq j$ and $j-i\leq i$.
    \item $i\geq j$ and $i-j\leq c-\left\lfloor\frac{j+1}{2}\right\rfloor$.
\end{enumerate}
\end{lemma}

\begin{proof}
    By Proposition~\ref{prop:premet},  the set of weights of $V$ is precisely the same as the set of weights of the corresponding $KG_K$-module
  $\Delta_K(\lambda)$. First we show that all weights satisfying either (i) or (ii) are weights of $\Delta_K(\lambda)$.

  Suppose $i\leq j\leq 2i$. Since $i+j\leq 2c$, we have that $i\leq c$. In particular, $\lambda-i0$ is a weight of $\Delta_K(\lambda)$.
  Now the weight $s_{\alpha_2}(\lambda-i0) = \lambda-i(2i)$ is also a weight of $\Delta_K(\lambda)$ and using \cite[VIII, \S7, Prop.~3]{bourbaki78} we have that for all
  $0\leq m\leq 2i$, $\lambda-im$ is a weight of $\Delta_K(\lambda)$. So in particular, $\lambda-ij$ is a
  weight of $V$.

  Suppose now that  $j\leq i \leq c+j- \lfloor\frac{j+1}{2}\rfloor$. As $i+j\leq 2c$, we have that $j\leq c$ and so
  $\lfloor\frac{j+1}{2}\rfloor\leq c$ and $\mu=\lambda-(\lfloor\frac{j+1}{2}\rfloor)\alpha_1$ is a weight of $V$. Hence,
  $s_{\alpha_2}(\mu) = \mu-2\lfloor\frac{j+1}{2}\rfloor\alpha_2$ is a weight of $\Delta_K(\lambda)$, which again by \cite[VIII, \S7, Prop.~3]{bourbaki78} implies that
  $\nu=\lambda-(\lfloor\frac{j+1}{2}\rfloor)j$ is a weight of $\Delta_K(\lambda)$. Further, we have that
  $\langle \nu,\alpha_1\rangle = c+j-2\lfloor\frac{j+1}{2}\rfloor$ and using again \emph{loc.cit.} we have that for all $0\leq m\leq c+j-2\lfloor\frac{j+1}{2}\rfloor$,
  $\nu-m\alpha_1 = \lambda-(m+\lfloor\frac{j+1}{2}\rfloor)j$ is a weight of $\Delta_K(\lambda)$, giving that $\lambda-ij$
  is a weight of $V$. 

  We now show that any weight $\lambda-ij$ with $i+j\leq 2c$ satisfies either (i) or (ii). To this end, we use the fact that the set of
  weights of $\Delta_K(\lambda)$ is the same as the set of weights of the module $V_1^{\otimes^c}$, the $c$-fold tensor product of the module $V_1$ with itself, where $V_1$ is the $KG_K$-module with highest
  weight $\omega_1$. (See  \cite[VIII, \S7, Prop.~10]{bourbaki78}.) Let $\mu$ be a weight of  $V_1^{\otimes^c}$, so that
  $\mu = c\omega_1-a_1\alpha_1-a_2(\alpha_1+\alpha_2)-a_3(\alpha_1+2\alpha_2)-a_4(2\alpha_1+2\alpha_2) = \lambda-(a_1+a_2+a_3+2a_4)\alpha_1-
  (a_2+2a_3+2a_4)\alpha_2$, with $a_i\in\mathbb N$ such that $a_1+a_2+a_3+a_4\leq c$.

  There are two cases to consider; suppose first that $a_1\leq a_3$, so that $a_1+a_2+a_3+2a_4\leq a_2+2a_3+2a_4$. Then
  $a_2+2a_3+2a_4\leq 2a_1+2a_2+2a_3+4a_4 = 2(a_1+a_2+a_3+2a_4)$ and the weight $\mu$ satisfies the conditions of (i).

  Now suppose $a_1\geq a_3$, so  that $a_1+a_2+a_3+2a_4\geq a_2+2a_3+2a_4$ and as usual
  \begin{equation}\label{ineq1} a_1+a_2+a_3+2a_4+ a_2+2a_3+2a_4 = a_1+2a_2+3a_3+4a_4\leq 2c, \end{equation} and \begin{equation}\label{ineq2} a_1+a_2+a_3+a_4\leq c.\end{equation} Note that $j-\lfloor\frac{j+1}{2}\rfloor = \lfloor\frac{j}{2}\rfloor$.
   If $a_1+a_2+a_3+2a_4 > c + \lfloor\frac{a_2+2a_3+2a_4}{2}\rfloor = c+a_3+a_4+\lfloor\frac{a_2}{2}\rfloor$, then 
   $a_1+a_2-\lfloor\frac{a_2}{2}\rfloor+a_4>c$ and $2a_1+a_2+1+2a_4>2c$. If $2a_1+2a_2+2a_3+2a_4\geq 2a_1+a_2+1+2a_4$ we obtain a
   contradiction to inequality (\ref{ineq2}). Hence we may now assume $2a_2+2a_3<a_2+1$, that is $a_2=0=a_3$. Now inequality (\ref{ineq2})
   becomes $a_1+a_4\leq c$ and so $a_1+2a_4\leq c+\lfloor\frac{2a_4}{2}\rfloor$ and the weight satisfies condition (ii).
\end{proof}

\begin{lemma}\label{B2 weights c1}
Let $\lambda = c1$, $c<p$. Then $\lambda-ij$, with $i+j\leq 2c$,  is a weight of $V$ if and only if one of the following holds.
\begin{enumerate}[label=\textnormal{(\roman*)}]
    \item $i\leq j$ and $j-i\leq i+1$.
    \item $i\geq j$ and $i-j\leq c-\left\lfloor\frac{j}{2}\right\rfloor$.
\end{enumerate}
\end{lemma}

\begin{proof}
    By \cite[VIII, \S7, Prop. 10]{bourbaki78} and Proposition~\ref{prop:premet}, the weights occurring in $V$ are the same as the weights occurring in $c0\otimes 01$. The statement then follows from Lemma~\ref{B2 weights c0}.
\end{proof}

\begin{lemma}\label{B2 c1c2 reduction}
    Let $\lambda = ab$ with $p>a\geq 1$, $p>b\geq 2$ and $2a+b+2 \equiv 0 \mod p$. Then $V\downarrow A$ is not MF.
\end{lemma}

\begin{proof}
    Since $2a+b+2 \equiv 0 \mod p$, by Lemma~\ref{B2 Weyl} we have \[\dim V = \dim L(ab) = \dim \Delta(ab)-\dim L((a-1)b)\geq  \dim \Delta(ab)-\dim \Delta((a-1)b).\] Using the Weyl character formula, we have that 
    \begin{flalign}\label{lower bound B2 ab eq}
 \dim V&\geq  \frac{1}{6} (1 + b) (6 + 6 a^2 + 5 b + b^2 + 12 a  + 6 a b).
 \end{flalign}
    Since $p>a$ and $p>b$, there are exactly two possibilities for $p$, either $p=2a+b+2$ and $b$ is odd, or $p=a+1+\frac{b}{2}$ and $b$ is even. Let us start with the first case, namely $p=2a+b+2$.
    Assume that $b=3$ and $a\geq 2$. Then $r=4a+3b=2p-1$ and \begin{flalign}\label{upper bound B2 ab eq b=3} B(r) = 2\sum_{k=1}^{\frac{p+1}{2}}(2k-1)+ \sum_{k=1}^{\frac{p-1}{2}}2k = \frac{3p^2+ 4p + 1}{4}. \end{flalign}
    Plugging in $p=2a+5$ and combining \eqref{upper bound B2 ab eq b=3} with \eqref{lower bound B2 ab eq} gives $\dim V > B(r)$ and Lemma~\ref{dimension bound comparison} implies that $V\downarrow A$ is not MF. The case $b=3$, $a=1$ and $p=7$ can be handled directly; we observe that $n_1=n_2=2$, while $n_3 = 4$, as the weight space $\lambda-12$ is $2$-dimensional (see \cite{lubeckOnline}). Then Lemma~\ref{sufficient conditions MF lemma}{(ii)} implies that $V\downarrow A$ is not MF. 
    
    Next assume that $b\geq 5$, in which case $r=2p +(b-4)<3p$. Then \begin{flalign}\label{upper bound B2 ab eq b>=5} B(r) = 3\sum_{k=1}^{\frac{b-3}{2}}2k+ 2\sum_{k=1}^{\frac{p+1}{2}}(2k-1)+ \sum_{k=1}^{\frac{p-1}{2}}2k = \frac{3p^2+ 4p + 10-12b+3b^2}{4}. \end{flalign} Plugging in $p=2a+b+2$ and combining \eqref{upper bound B2 ab eq b>=5} with \eqref{lower bound B2 ab eq} gives \[\dim V - B(r) \geq -39 - 36 a - 12 a^2 + 5 b + 6 a^2 b - 3 b^2 + 6 a b^2 + b^3 .\]
    As $b\geq 5$ and $a\geq 1$, this means that $\dim V-B(r)>0$, and Lemma~\ref{dimension bound comparison} implies that $V\downarrow A$ is not MF.
    
    We now consider the second case, where $p=a+1+\frac{b}{2}$. Here we have $r=4p+b-4$. Suppose that $b=2$, so that $a = p-2\geq 3$ and $r=3p+a<4p$. If $a$ is even, we have
    \begin{flalign}\label{upper bound B2 ab eq b=2 a even}
B(r) &= 4\sum_{k=1}^{\frac{a+2}{2}}(2k-1)+ 3\sum_{k=1}^{\frac{p-1}{2}}2k+2\sum_{k=1}^{\frac{p+1}{2}}(2k-1)+ \sum_{k=1}^{\frac{p-1}{2}}2k = &\\
 &= \frac{3p^2+ 2p + 7+8a+2a^2}{2}. \nonumber &
 \end{flalign}
     Plugging in $p=a+2$ and combining \eqref{upper bound B2 ab eq b=2 a even} with \eqref{lower bound B2 ab eq} gives \[\dim V - B(r) \geq \frac{1}{2}(a^2 + 2a-3) .\] Therefore $\dim V-B(r)>0$, and Lemma~\ref{dimension bound comparison} implies that $V\downarrow A$ is not MF. If $a$ is odd, we have
         \begin{flalign}\label{upper bound B2 ab eq b=2 a odd}
B(r) &= 4\sum_{k=1}^{\frac{a+1}{2}}2k+ 3\sum_{k=1}^{\frac{p+1}{2}}(2k-1)+2\sum_{k=1}^{\frac{p-1}{2}}2k+\sum_{k=1}^{\frac{p+1}{2}}(2k-1)= &\\
 &=\frac{3p^2+ 4p + 7+8a+2a^2}{2}. \nonumber &
 \end{flalign}
Plugging in $p=a+2$ and combining \eqref{upper bound B2 ab eq b=2 a odd} with \eqref{lower bound B2 ab eq} gives \[\dim V - B(r) \geq \frac{1}{2}(a^2 -7).\] As $a\geq 3$, Lemma~\ref{dimension bound comparison} implies that $V\downarrow A$ is not MF. 
    Now suppose that $b\geq 4$, in which case $r=4p+b-4<5p.$ Then
         \begin{flalign}\label{upper bound B2 ab eq b>=4}
B(r) &= 5\sum_{k=1}^{\frac{b-2}{2}}(2k-1)+4 \sum_{k=1}^{\frac{p-1}{2}}2k +3\sum_{k=1}^{\frac{p+1}{2}}(2k-1)+2\sum_{k=1}^{\frac{p-1}{2}}2k+\sum_{k=1}^{\frac{p+1}{2}}(2k-1)= &\\
 &=\frac{10p^2+8p+5b^2-20b+18}{4}. \nonumber &
 \end{flalign}
Plugging in $p=a+1+\frac{b}{2}$ and combining \eqref{upper bound B2 ab eq b>=4} with \eqref{lower bound B2 ab eq} gives \[\dim V - B(r) \geq \frac{1}{24}(-192 - 120 a - 36 a^2 + 80 b + 12 a b + 24 a^2 b - 21 b^2 + 
 24 a b^2 + 4 b^3) .\] We can write this as  \[\dim V - B(r) \geq \frac{1}{24}(-192 + 80 b - 21 b^2 + 4 b^3 + 12 a^2 (-3 + 2 b) + 
 12 a (-10 + b + 2 b^2)) .\] Treating the right-hand-side as a quadratic polynomial in $a$, it is easy to see that since $b\geq 4$, we must again have $\dim V-B(r)>0$, concluding by Lemma~\ref{dimension bound comparison}.
\end{proof}

\begin{lemma}\label{B2 1b general}
        Let $\lambda = 1b$ with $b\geq 2$. Then $V\downarrow A$ is not MF.
\end{lemma}

\begin{proof}
    By Lemma~\ref{B2 Weyl}, we have that one of the following holds. \begin{enumerate}[label=\textnormal{(\roman*)}]
    \item $p>b+5$ and $V = \Delta(1b)$.
    \item $b = p-4$.
    \item $b=p-2$ and $\dim V\geq \dim \Delta(1b)-\dim \Delta(1(b-2))$. 
    \end{enumerate}
    
    In the first case, for $b\geq 3$, $\dim V$ exceeds $B_K(r)$ and Lemma~\ref{dimension bound comparison} then implies that $V\downarrow A$ is not MF.  For $b=2$, 
    we have $p\geq 11>r$, contradicting our assumption that $p\leq r$.
    
    The second case is covered by Lemma~\ref{B2 c1c2 reduction}.
    
    Finally, we consider the third case. Here $\dim V \geq 2 (3 + 4 b + b^2) = 2(p-1)(p+1)$, $r = 3p-2>p$ and $b\geq 3$. In addition, we have  $B(r) =\frac{1}{2}(p+1)(3p-1)$ so that  $\dim V> B(r)$ and $V\downarrow A$ is not MF by Lemma~\ref{dimension bound comparison}.
\end{proof}

\begin{lemma}\label{B2 c1 not MF all 1 dim}
Let $\lambda = c1$ with $c\geq 1$ and $p=2c+3$. Then $V\downarrow A$ is not MF. 
\end{lemma}

\begin{proof}
Note that all $T$-weight spaces of $V$ are $1$-dimensional, see \cite[Prop. 2]{Za_Su_1dim}.
By Lemma~\ref{B2 weights c1}, for $1\leq d\leq c$ and $0\leq k\leq d$, we have that $\lambda-(d-k)k$ is a weight if and only if $0\leq k\leq 2d-2k+1$.
We then find that $n_d = \lfloor \frac{2d+1}{3}\rfloor+1$, for $1\leq d\leq c$. Since $\lambda-(c+1-k)k$ is a weight if and only if $1\leq k\leq 2c+2-2k+1$, we find that $n_{c+1}=\lfloor \frac{2c}{3}\rfloor+1$. Similarly $n_{c+2}=\lfloor \frac{2c+2}{3}\rfloor+1$.
There are now two cases to consider: either $c\equiv 1 \mod 3$ or $c\equiv -1 \mod 3$. In the first case $n_{c-1}>n_{c-2}$ (note that $n_{c-2}=0$ if $c=1$), and therefore by Lemma~\ref{sufficient conditions MF lemma}{(i)} we know that $(p+2)$ is a composition factor of $V\downarrow A$. We have $n_{c+2}=n_{c+1}+1$, and the $T_A$-weight $r-2(c+2) = p-4$ does not occur in the composition factor $(p+2)$, but the $T_A$-weight $p-2$ does. Therefore Lemma~\ref{sufficient conditions MF lemma}{(iv)} implies that $V\downarrow A$ is not MF. The second case, when $c\equiv -1\mod 3$, follows similarly.
\end{proof}

\begin{lemma}\label{B2 c1 not MF}
Let $\lambda = c1$ with $2c+3\neq p$ and $c\geq 1$. Then $V\downarrow A$ is not MF.
\end{lemma}

\begin{proof}
Note that $r=4c+3$. First assume  $c=1$, so that $r=7$; hence  $p=7$ and  the result follows from Lemma~\ref{CD lemma}. When $c=2$ we have $p\ne 7$ and by Lemma~\ref{B2 Weyl} $\Delta(\lambda)$ is irreducible. Since $p=5$ or $p=11$, the hypotheses of Lemma~\ref{CD lemma} are satisfied, and $V\downarrow A$ is not MF.

We henceforth assume that $c\geq 3$ and will show that $V\downarrow A$ is not MF.  Suppose first that $p>2c+4$, so that by Lemma~\ref{B2 Weyl} $\Delta(\lambda)$ is irreducible.  Since $p\geq 5$ and $r=4c+3$, the hypotheses of Lemma~\ref{CD lemma} are satisfied and $V\downarrow A$ is not MF.

Assume now that $p\leq 2c+4$, so that in fact $p\leq 2c+1$. Then by Lemma~\ref{B2 Weyl} and Remark~\ref{B2 Weyl remark}, either $p-3\leq c\leq p-1$ and $\Delta(\lambda)$ is irreducible and so $\dim V = \dim \Delta(\lambda)$; or $c\leq p-4$ and $\dim V\geq \dim \Delta(c1)-\dim \Delta((c-k)1)$, for $k=2c+4-p\geq 3$. In particular,  $\dim V\geq \dim \Delta(c1)-\dim \Delta((c-3)1) = 4 + 6 c + 6 c^2$. In both cases, one checks that $\dim V > B_{K}(r)$ so that $V\downarrow A$ is not MF by Lemma~\ref{dimension bound comparison}.
\end{proof}

\begin{lemma}\label{B2 c0 not MF p=2c+1}
Let $\lambda = c0$ with $c>1$ and $p=2c+1$. Then $V\downarrow A$ is MF if and only if $c=2$. 
\end{lemma}

\begin{proof}
    First note that $L(20)\downarrow A = (8)+(4)$, since here $p=5$. Now assume that $c\geq 3$, so that $p\geq 7$. By \cite[Prop. 2]{Za_Su_1dim}, all $T$-weight spaces of $V$ are $1$-dimensional. Let $1\leq l\leq c$. Then $n_l = \lfloor \frac{2l}{3}\rfloor+1$, since by Lemma~\ref{B2 weights c1} we have that $\lambda-(l-k)k$ is a weight if and only if $0\leq k\leq 2l-2k.$ Similarly $n_{c+1}=\lfloor \frac{2c+2}{3}\rfloor$, $n_{c+2}=\lfloor \frac{2c+1}{3}\rfloor$, $n_{c+3}=n_c$. Therefore by Lemma~\ref{sufficient conditions MF lemma}{(i)} either both $(p+1)$ and $(p+5)$ or both $(p+3)$ and $(p+5)$ are composition factors of $V\downarrow A$, and by Lemma~\ref{sufficient conditions MF lemma}{(iii)}, $(p-7)$ must occur as a repeated composition factor. 
\end{proof}

\begin{lemma}\label{B2 c0 not MF}
Let $\lambda = c0$ with $c>1$ and $p\neq 2c+1$. Then $V\downarrow A$ is not MF.
\end{lemma}

\begin{proof}
    First assume  that $p>2c+3$, so that Lemma~\ref{B2 Weyl} implies that $\Delta(c0)$ is irreducible. Since $r=4c$ and  by hypothesis $r\geq p$ and $p\geq 2c+4$, we have that $p$ does not divide $r+1$ and Lemma ~\ref{CD lemma} gives the result. 
    
    So we now assume that $p\leq 2c+3$. If $\Delta(\lambda)$ is irreducible, then one checks that $\dim V = \dim \Delta(\lambda)$ exceeds $B_{K}(r)$ for all $c\geq 9$. For $c\leq 8$ we combine the information from the tables  in \cite{lubeck} with the criteria of Lemma~\ref{CD lemma} to reduce to the case $c=5$ and $p=7$. But then we have $\dim V = 91$ and $B(r) = 88$, so we conclude by applying Lemma~\ref{dimension bound comparison}. 
    
    Now assume that $p\leq 2c+3$ and $\Delta(\lambda)$ is reducible. Then by Lemma~\ref{B2 Weyl}, $c+2<p$ and $2c+3 > p$,
    and so in particular, $4c = r >p$. By Remark~\ref{B2 Weyl remark}, $\dim V = \dim \Delta(c0) - \dim \Delta((c-k)0)$ for $k = 2c+3-p$, so $k$ is even. Assume that $c\geq k\geq 6$. We find that \[\dim V-B_{K}(r) = c^2 (-4 + k) - c (4 - 3 k + k^2) + 1/6 (-6 + 13 k - 9 k^2 + 2 k^3).\] Treating this as a quadratic polynomial in $c$, we find that $\dim V-B_{K}(r)$ is certainly strictly positive if $ -44 + 47 k - 16 k^2 + k^3>0$. Therefore if $k>12$, by Lemma~\ref{dimension bound comparison} we have that $V\downarrow A$ is not MF. If $k=6,8,10$ or $12$, we have $\dim V-B_{K}(r)>0$ when $c\geq 10$. Therefore the only possibilities for $(c,k)$, with $k\in\{6,8,10,12\}$,  are $(8,6)$ with $p=13$, $(7,6)$ with $p=11$, $(8,8)$ with $p=11$ or $(9,6)$ with $p=13$. In each of these cases, we find that $\dim V-B(r)>0$, concluding by Lemma~\ref{dimension bound comparison}.
    
    Note that $k \ne 2$ as $p\ne 2c+1$.  So finally we consider the case $k=4$ and $p=2c-1$. Now $r = 4c = 2p+2$ and a direct computation shows that $\dim V = 4c^2-4c+6$ while $B(r) = 3c^2-2c+12$. Now we have $c\geq 4$ (since $k=4$) and hence $\dim V$ exceeds $B(r)$, showing as before that $V\downarrow A$ is not MF. 
\end{proof}

    \begin{lemma}\label{B2 0c not MF}
Let $\lambda = 0c$ with $c>1$ and $p\leq r$. Then $V\downarrow A$ is not MF.
\end{lemma}

\begin{proof} We have $V = \Delta(0c)$ (see \cite[Table 1]{seitz_mem_class}). One checks that $\dim V = \frac{1}{6}(1 + c) (2 + c) (3 + c)> B_{K}(r)$ for $c\geq 9$; by Lemma~\ref{dimension bound comparison}, the module $V\downarrow A$ is not MF in these cases. For $2\leq c\leq 8$, we may apply Lemma~\ref{CD lemma} to conclude that $V\downarrow A$ is not MF except for the pairs $(c,p) = (3,5)$ and $(c,p) = (7,11)$. Here we apply Lemma~\ref{dimension bound for omega_i} to again conclude that $V\downarrow A$ is not MF.
\end{proof}

\renewcommand*{\proofname}{Proof of Proposition~$\ref{B2 classification prop}$.}
\begin{proof}
Suppose that $V\downarrow A$ is MF, with $\lambda = ab$ and $r\geq p$. By Proposition~\ref{first reduction lemma} and Lemma~\ref{1 dimensional weight space lemma}, if $a,b\geq 2$ we must have $ 2a+b+2\equiv 0\mod p$. Therefore Lemma~\ref{B2 c1c2 reduction} implies that either $a\leq 1$ or $b\leq 1$ and  Lemma~\ref{B2 c1 not MF all 1 dim} and Lemma~\ref{B2 c1 not MF} show that $\lambda \ne 11$. By Lemma~\ref{B2 1b general} we conclude that if $a=1$, then $b=0$ contrary to our assumption that $r\geq p$ and  another application of Lemma~\ref{B2 c1 not MF all 1 dim} and Lemma~\ref{B2 c1 not MF} shows that if $b = 1$ then $a=0$, again contrary to our assumption on $p$ and $r$. We therefore reduce to the case $a=0$ or $b=0$, the first being ruled out by Lemma~\ref{B2 0c not MF}. If $b=0$, by Lemma~\ref{B2 c0 not MF p=2c+1} and Lemma~\ref{B2 c0 not MF}, and the above remarks,  we conclude that $a=2$ with $ p=5$, in which case $V\downarrow A$ is MF by Lemma~\ref{B2 c0 not MF p=2c+1}.
\end{proof}
\renewcommand*{\proofname}{Proof.}

\subsection{The case where $G$ is $G_2$}
We now move on to the final case where $G$ has rank $2$, i.e. $G=G_2$. Our main result, to be proven in a sequence of lemmas,  is the following proposition.

\begin{prop}\label{G2 classification main prop} Let $G = G_2$ and $\lambda = ab$ with $p\leq r$. Then $V\downarrow A$ is not MF.
\end{prop}

Set $\lambda = ab$, with $0\leq a,b<p$, where we take $\alpha_1$ to be short, $(\alpha_2,\alpha_2) = 1$. (This choice of root lengths is required for using the result \cite[(6.2)]{seitz_mem_class} stated below in Lemma~\ref{compfactor_Weyl}.) Here we have $r = 6a+10b$, and $p\geq 7$ since $p\geq h$.  In addition we set $\mu = \lambda-11$ throughout the entire section and note that $\mu = (a+1)\omega_1+(b-1)\omega_2$. We require an additional notation: for $\alpha\in\Phi$, we let $e_\alpha, f_\alpha$ denote the $T$-weight vectors in the Lie algebra of $G$ associated with the root $\alpha$, respectively $-\alpha$. 

We will use a result from \cite{seitz_mem_class}, which we state here only for the group $G_2$:

\begin{lem}\textup{\cite[(6.2)]{seitz_mem_class}}\label{compfactor_Weyl}  Assume $p>3$. Let $\nu$ be a dominant weight such that $L(\nu)$ affords a composition factor of $\Delta(\lambda)$. 
Then $$2(\lambda+\rho,\lambda-\nu)-(\lambda-\nu,\lambda-\nu)\in\frac{p}{6}\mathbb Z.$$
\end{lem}

In view of applying Lemma~\ref{compfactor_Weyl}, we record the results of some computations for particular subdominant weights in $\Delta(\lambda)$ in Table~\ref{tab: weights G2}.
\begin{center}
\begin{xltabular}[h]{\textwidth}{c c c c}
\caption{Weight multiplicities for $G_2$-modules} \label{tab: weights G2} \\
\hline \multicolumn{1}{c}{\textbf{$(a,b)$}} & \multicolumn{1}{c}{$\nu$} & \multicolumn{1}{c}{$\dim \Delta(\lambda)_\nu$} & \multicolumn{1}{c}{$2(\lambda+\rho,\lambda-\nu)-(\lambda-\nu,\lambda-\nu)$}  \\ \hline 
\endhead
$a\geq 1, b\geq 1$&$\lambda-21$ &$3-\delta_{a,1}$&$\frac{2a+3b+4}{3}$\\ 
$a\geq 1, b\geq2$&$\lambda-12$&$2$&$\frac{a+6b}{3}$\\ 
$a\geq 1, b\geq1$&$\lambda-22$&$4-\delta_{a,1}-\delta_{b,1}$&$\frac{2a+6b+4}{3}$\\
$a=0, b\geq2$&$\lambda-22$&$2$&$\frac{6b+4}{3}$\\
$a\geq 1, b\geq2$&$\lambda-13$&$2-\delta_{b,2}$&$\frac{a+9b-9}{3}$\\ 
$a\geq 1, b\geq2$&$\lambda-32$&$7-2\delta_{a,1}-\delta_{a,2}$&$a+2b+2$\\ 
$a\geq 1, b\geq3$&$\lambda-23$&$4-\delta_{a,1}$&$\frac{2a+9b-2}{3}$\\ 
$a=1, b\geq3$&$\lambda-14$&$2-\delta_{b,3}$&$\frac{a+12b-24}{3}$\\ \hline
\end{xltabular}
\end{center} 

In addition, we note that since $\lambda$ is $p$-restricted, $V$ is irreducible as a module for the Lie algebra of $G$ (see \cite[Cor. Ch. 1]{curtis}). For the following lemmas, we let $v^+\in V_\lambda$, that is,  a highest weight vector in $V$. Then by \cite[1.29]{Tes_mem_exc} we have that. For $\nu\leq\lambda$, the weight space $V_\nu$ is spanned by vectors of the form $f_{\gamma_1}^{m_1}\cdots f_{\gamma_r}^{m_r}v^+$, where $\gamma_j\in\Phi^+$ and $m_j\in\mathbb N$ with $\lambda-\nu=\sum m_j\gamma_j$.

\begin{lem}\label{lemma G2 21 dimension} Assume $a\geq 2$, $b\geq 1$ and set $\nu = \lambda-21$. The following hold.  \begin{enumerate}[label=\textnormal{(\roman*)}]
  \item If $a+3b+3\not\equiv 0\mod p$ and $2a+3b+4\not\equiv 0\mod p$, then $\dim V_\nu = 3$.
  \item If $a+3b+3\equiv 0 \mod p$, then $\dim V_\nu = 2$.
  \item If $2a+3b+4\equiv 0 \mod p$, then $\dim V_\nu = 2$. 
    \end{enumerate}

\end{lem}

\begin{proof} Using \cite{lubeckOnline} and \cite[Proposition A]{cavallin}, one checks that $\dim \Delta(\lambda)_\nu = 3$. Note as well that if $\eta$ is a  dominant weight satisfying  $\nu\prec \eta\prec \lambda$, then $\eta\in\{\lambda-10,\lambda-01\  ({\text {if }} b\geq 2), \lambda-20\  ({\text{if }} a\geq 4), \mu\}$. The weight $\eta$ does not afford a composition factor of $\Delta(\lambda)$, for $\eta\in X(T)$, $\eta\ne \mu$.

First consider the case where $a+3b+3\not\equiv 0\mod p$. In this case, $\mu$ does not afford a composition factor of $\Delta(\lambda)$ (see Lemma~\ref{1 dimensional weight space lemma}) and the
  vectors $f_{\alpha_1+\alpha_2}v^+$ and $f_{\alpha_2}f_{\alpha_1}v^+$ are linearly independent.
  The weight space $V_\nu$ is spanned by the vectors  $v_1 = f_{2\alpha_1+\alpha_2}v^+$, $v_2 = f_{\alpha_1+\alpha_2}f_{\alpha_1}v^+$ and
  $v_3 = f_{\alpha_2}f_{\alpha_1}^2v^+$. Suppose $\sum_{i=1}^3 a_iv_i=0$, for $a_i\in \mathcal{k}$. Then applying $e_{\alpha_1}$ and $e_{\alpha_2}$ respectively, and using the fact that $f_{\alpha_1}^2v^+\ne 0$, we obtain the following system of equations.

  $$2a_1+a a_2 = 0,\quad 3a_2+a_3(2a-2) = 0\mbox { and } a_3(b+2)-a_2=0.$$

 (These computations depend on a choice of structure constants; we have used those given in \cite[\S12.5]{cartersimple}.)  We then have that  $v_1,v_2,v_3$ are linearly dependent if and only if $a_3\ne 0$. If $a_3\ne 0$, then we deduce that $2a+3b+4\equiv 0\mod p$. Moreover, if $2a+3b+4\equiv0\mod p$ the three vectors are linearly dependent and it is easy to check that $v_1$ and $v_2$ are linearly independent. This gives (i).

  Now consider the case where $a+3b+3\equiv 0\mod p$, so that $\mu$ affords a composition factor of $\Delta(\lambda)$ and one checks that $bf_{\alpha_1+\alpha_2}v^+ +f_{\alpha_1}f_{\alpha_2}v^+ = 0$. Now if $a=p-1$ (so that $2a+3b+4\equiv 0\mod p$), then $\nu$ does not occur in the composition factor afforded by $\mu$. In addition, arguing as above, we see that
  $v_1\in\langle v_2,v_3\rangle$ and $v_2$ and $v_3$ are linearly independent, so that $\dim V_\nu = 2$. While if $a\ne p-1$, then $\nu$ occurs in the
  composition factor afforded by $\mu$, with multiplicity 1. Moreover,  $2a+3b+4\not\equiv 0\mod p$, and Lemma~\ref{compfactor_Weyl} implies that the weight $\nu$ does not afford a composition factor of $\Delta(\lambda)$ and so $\dim V_{\nu} = 2$. These arguments give the conclusions of (ii) and (iii).\end{proof}

\begin{lem}\label{lemma g2 21 dimension a=1} Let $a=1$, $b\geq 1$,  and set $\nu = \lambda-21$. Then $\dim V_\nu = 1$ if $3b+4\equiv 0\mod p$ and $\dim V_\nu = 2$ otherwise.

\end{lem}

\begin{proof} Using \cite{lubeckOnline} and \cite[Proposition A]{cavallin}, we have $\dim \Delta(\lambda)_\nu = 2$ and $V_\nu$ is spanned by $v_1 = f_{2\alpha_1+\alpha_2}v^+$ and $v_2 = f_{\alpha_1+\alpha_2}f_{\alpha_1}v^+$. If $3b+4\equiv 0\mod p$, then by Lemma~\ref{1 dimensional weight space lemma}, $\mu$ affords a composition factor of $\Delta(\lambda)$ and $\nu$ occurs with multiplicity 1 there. So by Proposition~\ref{prop:premet}, we have $\dim V_\nu = 1$.

  If $3b+4\not\equiv0\mod p$, then $\mu$ does not afford a composition factor and  $f_{\alpha_1+\alpha_2}v^+$ and $f_{\alpha_2}f_{\alpha_1}v^+$ are linearly independent. If $a_1v_1+a_2v_2 = 0$ for $a_i\in \mathcal{k}$, then applying $e_{\alpha_1}$ and $e_{\alpha_2}$, we deduce that $2a_1+a_2 = 0 = 3a_2$. Hence the two vectors are linearly independent and $\dim V_\nu = 2$.\end{proof}

\begin{lem}\label{lemma G2 dimension 12} Assume $b\geq 2$, $a\geq 1$, and set $\nu = \lambda-12$. Then $\dim V_\nu = 1$ if $a+3b+3\equiv 0\mod p$ and $\dim V_\nu = 2$ otherwise.

\end{lem}

\begin{proof} As in the preceding lemmas, we find that $\dim \Delta(\lambda)_{\nu} = 2$. If $a+3b+3\equiv 0\mod p$, then $\mu$ affords a composition factor of $\Delta(\lambda)$ and using Proposition~\ref{prop:premet} we deduce that $\dim V_\nu = 1$.

  So assume that $a+3b+3\not\equiv 0\mod p$ and then $f_{\alpha_1+\alpha_2}v^+$ and $f_{\alpha_2}f_{\alpha_1}v^+$ are linearly independent. The $\nu$ weight space is spanned by $v_1 = f_{\alpha_1+\alpha_2}f_{\alpha_2}v^+$ and $v_2 = f_{\alpha_2}^2f_{\alpha_1}v^+$.  Suppose $a_1v_1+a_2v_2 = 0$ for $a_i\in \mathcal{k}$. Applying
  $e_{\alpha_1}$ and $e_{\alpha_2}$ and using that $f_{\alpha_2}^2v^+\ne 0$, we deduce that $3a_1+aa_2 = 0$ and $a_1b = 0$. Hence the two vectors are linearly independent, giving the result. \end{proof}

We are now ready to prove the main proposition.
\renewcommand*{\proofname}{Proof of Proposition~$\ref{G2 classification main prop}$.}
\begin{proof}

  We  treat various cases separately below. In Cases 1 to 4, we use Proposition~\ref{first reduction lemma}(v) and Lemma~\ref{1 dimensional weight space lemma} to reduce to the case where $a+3b+3\equiv 0\mod p$ (as otherwise $V\downarrow A$ is not MF and neither is $\Delta_K(\lambda)\downarrow A_{K}$). Throughout we rely on the Tables in \cite{lubeckOnline}.

{{Case 1:}} $a\geq 3$ and $b\geq 3$. The $T_A$-weight $r-6$ is
           afforded by $\lambda-30$, $\lambda-03$, $\lambda - 21$ and $\lambda-12$. An application of Lemma~\ref{lemma G2 21 dimension}  then shows that $n_3\geq 5$ and Lemma~\ref{multiplicity bound general lemma} then shows that $V\downarrow A$ is not MF.
           
           {{Case 2:}} $a=2$ and $b\geq 4$. Here $3b+5\equiv 0\mod p$, so Lemma~\ref{compfactor_Weyl} implies that neither of the weights $\lambda-21$ and $\lambda - 12$ affords a composition factor of $\Delta(\lambda)$. Now, the $T_A$-weight $r-8$ is afforded by
           $\lambda-31$, $\lambda - 22$, $\lambda-13$ and $\lambda-04$. By Lemma~\ref{compfactor_Weyl}, none of these weights affords a composition factor of $\Delta(\lambda)$.  Therefore $\dim V_{\lambda-31} = \dim V_{\lambda-22}=2$, Indeed, for the weight $\lambda -31$ we note that the only dominant weights $\mu'$ with $\lambda-31\prec\mu'$ are $\mu$, 
           $\lambda-01$, $\lambda-10$, $\lambda-21$. We have assumed that $\mu$ affords a composition factor of $\Delta(\lambda)$, but the second and third weights occur with multiplicity 1 in $\Delta(\lambda)$ and so do not afford a composition factor of $\Delta(\lambda)$ (and as mentioned above, neither does the fourth weight). This then allows us to determine $\dim V_{\lambda-31}$ and similarly for $V_{\lambda-22}$. We conclude that $n_4\geq 6$ and apply Lemma~\ref{multiplicity bound general lemma} to see that $V\downarrow A$ is not MF.
           
           {{Case 3:}} $a\geq 3$ and $b=2$. Here we have $a+9\equiv 0 \mod p$. Lemma~\ref{compfactor_Weyl} implies that neither of the weights $\lambda-21$, $\lambda-12$ affords a composition factor of $\Delta(\lambda)$. Now consider the $T_A$-weight $r-8$,
           afforded by $\lambda-31$, $\lambda-22$ and $\lambda - 13$, none of which affords a composition factor of $\Delta(\lambda)$. Counting the occurrences of these weights in the irreducible $L(\mu)$,
           we see that $n_4\geq 6$ and then use Lemma~\ref{multiplicity bound general lemma} to see that $V\downarrow A$ is not MF.

          {{Case 4:}} $a = 2$ and $b\in \{2,3\}$. Here, we have  $a+3b+3\equiv 0\mod p$.
          Consider first the weight $\lambda = 2\omega_1+2\omega_2$ with $p=11$; here $\dim V = 295$ and $r=32$.
          One then checks that $B(r) = 204$. 
          For the weight $\lambda = 2\omega_1+3\omega_2$, with $p=7$, we have $r=42$, $\dim V = 532$ and $B(r) = 295$. In both cases, Lemma~\ref{dimension bound comparison} then implies that $V\downarrow A$ is not MF.\smallbreak

          We now turn to the cases where one or both of $a$ and $b$ is less than $2$, in which case we no longer deduce that 
          $\dim V_\mu=1$.\smallbreak

           {{Case 5:}} $a\geq 3$ and $b=1$. If $a=p-6$, then $n_2=2$, while the $T_A$-weight $r-6$ is afforded by $\lambda-21$, $\lambda-12$ and $\lambda-30$; using Lemma~\ref{compfactor_Weyl} we have that $n_3=4$ and so $V\downarrow A$ is not MF by Lemma~\ref{sufficient conditions MF lemma}. 

           Now suppose $a\ne p-6$ so that $n_2=3$ and $\mu$ does not afford a composition factor of $\Delta(\lambda)$. Let $\nu = \lambda-21$. Suppose first that $2a+7\equiv 0\mod p$; then
           by Lemma~\ref{lemma G2 21 dimension} we have $\dim V_{\nu} = 2$, which implies that $[\Delta(\lambda):L(\nu)] = 1$ and $n_3=4$. Now count the occurrences of the $T_A$-weight $r-8$ which is afforded by $\lambda-31$, $\lambda-22$ and $\lambda-40$, the latter only
           if $a\geq 4$. If $a\geq 4$, Lemma~\ref{compfactor_Weyl} implies that $n_4\geq 6$, giving the usual contradiction. The case where $2a+7\not\equiv 0\mod p$ is easier; here $\nu$ does
           not afford a composition factor of $\Delta(\lambda)$ and $n_3=5 = n_2+3$ (even if $a=3$).

           So we are left with the case $a=3$, $b=1$ and $p=13$, where $\dim V = 259$ and $r=28$. But as above, one checks that $\dim V >B(r)$,  and Lemma~\ref{dimension bound comparison} implies that $V\downarrow A$ is not MF.

           {{Case 6:}} $a=1$ and $b\geq 3$.  Consider first the case where $3b+4\equiv0\mod p$, when $\mu$ affords a composition factor of $\Delta(\lambda)$. Moreover, we note that $b\ne 4$. We claim that $n_4 = 4-\delta_{b,3}$ and $n_5\geq 6-\delta_{b,3}$, which then shows that $V\downarrow A$ is not MF.

           The $T_A$-weight $r-8$ is afforded by $\lambda - 31$, $\lambda - 22$, $\lambda - 13$ and $\lambda - 04$ (the latter only if $b\geq 4$). The first of these is conjugate to $\mu$ and so has multiplicity 1 in $V$ and the last of these has multiplicity $1-\delta_{b,3}$. For the remaining two weights, we use repeatedly Lemma~\ref{compfactor_Weyl} and note that \begin{enumerate}[label=\textnormal{(\roman*)}]
               \item $\lambda-21$ and $\lambda-12$ do not afford composition factors of $\Delta(\lambda)$; 
               \item $\lambda - 22$ does not afford a composition factor of $\Delta(\mu)$ and so occurs with multiplicity 2 in $L(\mu)$; and
               \item $\lambda - 13$ and $\lambda - 22$ do not afford composition factors of $\Delta(\lambda)$. 
    \end{enumerate}
           We then deduce that the weights $\lambda-22$ and $\lambda - 13$ each occur with multiplicity 1 in $V$. Hence $n_4 = 4-\delta_{b,3}$ as claimed.

           Now we turn to $n_5$; the $T_A$-weight $r-10$ is afforded by $\lambda - 41$, $\lambda - 32$, $\lambda  - 23$, $\lambda - 14$ and $\lambda - 05$ (the latter only if $b\geq 5$). The first of these is conjugate to $\lambda$. We now argue that $\nu = \lambda - 32$ has multiplicity $2$ in $V$, which establishes the claim on $n_5$.  Note that $\nu$ does not afford a composition factor of $\Delta(\lambda)$ nor of $\Delta(\mu)$.
         Applying Lemma~\ref{lemma G2 21 dimension}, we deduce that $\nu = \mu-21$ has multiplicity 3 in $L(\mu)$ and so has multiplicity $2$ in $V$, as claimed.

           Now consider the case where $3b+4\not\equiv 0\mod p$ and so $n_2 = 3$. By Lemma~\ref{lemma g2 21 dimension a=1} and Lemma~\ref{lemma G2 dimension 12} we have $\dim V_{\lambda-21} = \dim V_{\lambda-12}=2$, which means that $n_3 = 5$, so that $V\downarrow A$ is not MF.

           {{Case 7:}} $(a,b)\in\{(1,1), (1,2), (2,1)\}$. Here we have $r = 16$, respectively $26$, $22$. If $p\ne 7$, respectively $p\ne 7$, $p\ne 11$, the Weyl modules are irreducible and we may apply Lemma~\ref{CD lemma}. For the primes $p=7,7, 11$, respectively, an application of Lemma~\ref{dimension bound for omega_i}  shows that $V\downarrow A$ is not MF in the second and third cases. Now for the case $\lambda = \omega_1+\omega_2$ and $p=7$, we must argue more carefully. Here, one checks that  the weights $r,r-2,r-4,r-6$ occur with multiplicities $1,2,1,2$ respectively. Since $r-6$ does not occur as a weight in $(r)$, while $r-4$ does,  by Lemma~\ref{sufficient conditions MF lemma} we conclude that $V\downarrow A$ is not MF.

         {{Case 8:}} $b=0$. Here we view $G$ as a subgroup of $B_3$ via the $7$-dimensional irreducible representation afforded by $L(\omega_1)$. Then we have that $A\subset G$ is the principal $A_1$-subgroup of $B_3$ and moreover the $B_3$-module $L_{B_3}(a\omega_1)$ remains irreducible upon restriction to $G$, and affords the module $V$. (See \cite{seitz_mem_class}, Table 1.)  Hence, we can use the $B_3$ analysis, which is given in Proposition~\ref{rank>=3 single support main proposition}, to conclude.

         {{Case 9:}} $a=0$, $b\geq 4$. Here the $T_A$-weight $r-6$ is afforded by $\lambda - 21$, $\lambda - 12$ and $\lambda - 03$, each of which has multiplicity $1$ in
         $\Delta(\lambda)$ and so $n_3=3$. In particular, none of the listed weights affords a composition factor of $\Delta(\lambda)$, nor does $\lambda-11$.  Now we separate into two cases. First suppose that  $\lambda- 22$ does not
         afford a composition factor of $\Delta(\lambda)$; then $n_4\geq 5$ and $V\downarrow A$ is not MF.

         Now suppose that $\nu = \lambda-22$ affords a composition factor of $\Delta(\lambda)$ and so by Lemma~\ref{compfactor_Weyl} we have $3b+2\equiv0\mod p$. We first treat the case where $b\geq 6$.
         We claim that $n_5 = 4$. The $T_A$-weight $r-10$  is afforded by $\lambda - 23$, $\lambda - 32$, $\lambda - 14$, and $\lambda - 05$. The first two occur in the composition factor afforded by $\nu$, each with multiplicity 1,  and using the multiplicities in the Weyl module and Proposition~\ref{prop:premet}, we see that each of the four listed weights occurs with multiplicity 1 in $V$, establishing the claim. The $T_A$-weight $r-12$ is afforded by $\lambda - 42$, $\lambda - 33$,
         $\lambda - 24$, $\lambda - 15$ and $\lambda - 06$. The second of these weights has multiplicity $4$ in $\Delta(\lambda)$ and
         occurs with multiplicity 2 in $L(\nu)$. Moreover, this weight does not afford a composition factor of $\Delta(\lambda)$ (nor does any dominant weight $\eta\ne \nu$ with $\lambda-33\prec \eta\prec\lambda$) and so occurs with multiplicity 2 in $V$. Hence, $n_6\geq 6$  and $V\downarrow A$ is not MF.

        It remains to consider the cases $b=4$ and $b=5$ with $p=7$, respectively $p=17$ and $\dim V = 267$, respectively 546. In both cases, an application of Lemma~\ref{dimension bound for omega_i} shows that $V\downarrow A$ is not MF.

        {{Case 10:}} $a=0$ and $1\leq b\leq 3$. When $b=1$, the Weyl module is irreducible and the result follows from Lemma~\ref{CD lemma}. If $b=2$ and $p\ne 7$, we may apply Lemma~\ref{CD lemma} to conclude. When $(b,p) = (2,7)$, we use Lemma~\ref{n_d same as p=0 when Weyl module irreducible lemma} and the proof of \cite[Lemma 4.5]{LSTA1} to deduce that $n_0=1, n_1=1, n_2=2, n_3=2, n_4=3, n_5=4$ , so that $V\downarrow A$ has composition factors $(20), (16)$ and $(12)$. Since the $T_A$-weight $12$ lies in the composition factor $(16)$ but the $T_A$-weight $10$ does not, Lemma~\ref{sufficient conditions MF lemma} implies that $V\downarrow A$ is not MF. 
        
        Finally, we consider the case $b=3$, where $r=30$. Here the Weyl module is irreducible unless $p=11$. If $p\ne 11$, the result follows from Lemma~\ref{CD lemma}. If $p=11$, we use \cite{lubeckOnline} to see that $n_0=1, n_2=1, n_2=2, n_3=3, n_4=3, n_5=3$. We then deduce that $V\downarrow A$ has no composition factor $(22)$, nor $(20)$. But then $\dim V = 148 > B(30)-\dim(20)-\dim(22)$, so that $V\downarrow A$ is not MF.\end{proof}
\renewcommand*{\proofname}{Proof.}

\section{The case where $G$ has rank at least $3$}\label{rank at least 3 section}
In this section we handle the case where $G$ has rank at least $3$, establishing the following proposition. 

\begin{proposition}\label{main prop rank >=3}
    Suppose that $G$ has rank at least $3$ and $p\leq r$. Then $V\downarrow A$ is not MF.
\end{proposition}

We assume throughout Section 4 that $p\leq r$. Furthermore by Lemma~\ref{first reduction lemma}{(i)} we only need to consider the case $\lambda = c_i\omega_i+c_j\omega_j$ (with $c_i$ or $c_j$ possibly $0$), i.e. the weight $\lambda$ has support on at most two nodes.
\subsection{The case \(c_ic_j\) not \(0\)} 
We treat the case where $\lambda = c_i\omega_i+c_j\omega_j$ with $c_ic_j\neq 0$ in a sequence of lemmas. 

\begin{lemma}\label{rank at least 4 adjacent reduction}
Suppose that $G$ has rank at least $4$ and $\lambda = c_i\omega_i+c_j\omega_j$ with $\alpha_i$ and $\alpha_j$ adjacent and $c_i,c_j\geq 1$. Then $V\downarrow A$ is not MF.
\end{lemma}

\begin{proof}
    Since $p\geq h$ we have $p\geq 5,11,11,7$ respectively for $G=A_\ell,B_\ell,C_\ell,D_\ell$ and $p\geq 13,13,19,31$ respectively for $G=F_4,E_6,E_7,E_8$.
    By Proposition~\ref{first reduction lemma}{(iv)} and {(v)}, we can assume that $c_i=1$ or $c_j=1$, and $\alpha_i$ or $\alpha_j$ is an end-node. Recall that by Proposition~\ref{prop:premet}, the set of weights of $V$ is the same as the set of weights of $\Delta(\lambda)$. Using this, it is straightforward to see that if $c_i\geq 3$ or $c_j\geq 3$, then $V\downarrow A$ is not MF. For example if $\lambda = c_1 \omega_1 +\omega_2$ (so $G$ is not of type $E_\ell$) and $c_1\geq 3$, the $T_A$-weight $r-6$ is afforded by $\lambda-123$, $\lambda-234$, $\lambda-12^2$, $\lambda-1^22$ and $\lambda -1^3$. Therefore $n_3\geq 5$, and Lemma~\ref{multiplicity bound general lemma} implies that $V\downarrow A$ is not MF. Similarly, if $\lambda=\omega_1+c_2\omega_2$, with $c_2\geq 3$ (so again $G$ is not of type $E_\ell$), then the $T_A$-weight $r-6$ is afforded by $\lambda-123$, $\lambda-234$, $\lambda-12^2$, $\lambda-2^23$ and $\lambda -2^3$. As before, $n_3\geq 5$, and  $V\downarrow A$ is not MF. If $c_i = 2$ or $c_j = 2$, by Lemma~\ref{1 dimensional weight space lemma} we have $ \dim V_{\lambda-\alpha_i-\alpha_j}>1$, and by Proposition~\ref{first reduction lemma}{(vi)} the module $V\downarrow A$ is not MF. Thus, we reduce to $c_i=c_j=1$.
    
    Consider $\lambda = \omega_1+\omega_2$. For $G$ classical, the weights $\lambda-123 = (\lambda-12)^{s_3},\lambda-234,\lambda-1^22=(\lambda-2)^{s_1},\lambda-12^2=(\lambda-1)^{s_2}$ occur with multiplicities $2$, $1$, $1$, $1$ by Lemma~\ref{1 dimensional weight space lemma}. Therefore $n_3\geq 5$ and Lemma~\ref{multiplicity bound general lemma} implies that $V\downarrow A$ is not MF. The same argument, with the appropriate relabelling of indices, handles all remaining cases where $(\alpha_i,\alpha_i) = (\alpha_j,\alpha_j)$, including the cases $G=E_\ell$ and $\lambda \in\{\omega_1+\omega_3,\omega_2+\omega_4,\omega_{\ell-1}+\omega_\ell\}$. For the group of type $F_4$ and the weights $\omega_1+\omega_2$ and $\omega_3+\omega_4$, we use the weight space dimensions provided in \cite{lubeckOnline} to conclude again that $n_3\geq 5$.

    Therefore we reduce to $G=B_\ell$ or $G=C_\ell$ with $\lambda = \omega_{\ell-1}+\omega_\ell$. Suppose $G=B_\ell$. Since $p\geq h$, by Lemma~\ref{1 dimensional weight space lemma} we have $\dim V_{\lambda-(\ell-1)\ell} =  2$. The $T_A$-weight $r-6$ is afforded by $\lambda-(\ell-1)\ell^2=(\lambda-(\ell-1)\ell)^{s_\ell}$, $\lambda-(\ell-2)(\ell-1)\ell$ and $\lambda-(\ell-1)^2\ell$. If $\ell = 4$, the first two weight spaces have dimension $2$ by \cite{lubeckOnline}. By Lemma~\ref{subdiagram lemma}, for any $\ell\geq 4$, we have $n_3\geq 5$, and Lemma~\ref{multiplicity bound general lemma} implies that $V\downarrow A$ is not MF. The $C_\ell$ case is handled similarly.
\end{proof}

\begin{lem}\label{A3 c10 and 1c0 reduction}
 Let $G=A_3$ and $\lambda = c\omega_1+\omega_2$ or $\omega_1+c\omega_2$, with $c>1$. Then $V\downarrow A$ is not MF.
\end{lem}

\begin{proof}
By Proposition~\ref{first reduction lemma}{(vi)} we can assume that the weight space $\lambda -12$ is $1$-dimensional. In particular we must have $ c = p-2$ by Lemma~\ref{1 dimensional weight space lemma}. Let us start with $\lambda = \omega_1+(p-2)\omega_2$. Since $p\geq 5$, the $T_A$-weight $r-6$ is afforded by $\lambda - 123$, $\lambda-12^2$, $\lambda-2^23$, $\lambda-2^3$ and $\lambda-1^2 2$. Therefore Lemma~\ref{multiplicity bound general lemma} implies that $V\downarrow A$ is not MF.

For the case $\lambda = (p-2)\omega_1+\omega_2$, we will use a dimension argument. We refer to the discussion in \cite[Part II, 8.20]{Jantzen}, where the weight
  $\lambda$ satisfies the conditions of the weight $\lambda_2$, with $s=1=t$ and $r=p-3$. Then one has
  \[\dim V = \dim \Delta(\lambda) - \dim \Delta(\lambda - \alpha_1-\alpha_2) + \dim \Delta(\lambda - 2\alpha_1-2\alpha_2-\alpha_3).\]

  Using the Weyl degree formula we find that $\dim V = \frac{(p-1)}{6}(p^2+7p+18)$.
  We have $r = 3p-2$, and a simple calculation shows that $B(r) = \frac{(p+1)(3p-1)}{2}$. Since $B(r)<\dim V$ for all $p>3$, by Lemma~\ref{dimension bound comparison} we conclude that $V\downarrow A$ is not MF.
\end{proof}

\begin{lem}\label{B3 and C3 c1 reduction}
    Let $G=B_3$ or $C_3$ and let $\lambda \in \{c \omega_1+\omega_2$, $\omega_1+c\omega_2$, $c\omega_2+\omega_3$, $\omega_2+c\omega_3\}$, with $c>1$. Then $V\downarrow A$ is not MF.
\end{lem}

\begin{proof}
By Proposition~\ref{first reduction lemma}{(vii)} we can assume that the weight space $\lambda -ij$ is $1$-dimensional, where $\lambda = c_i\omega_i+c_j\omega_j$. Note that $p\geq 7$ as $p\geq h$.

Case 1 : $\lambda = c\omega_1+\omega_2$. As the weight space $\lambda -ij$ is $1$-dimensional, we have $c = p-2$ by Lemma~\ref{1 dimensional weight space lemma}. In particular $c\geq 5$. The $T_A$-weight $r-6$ is afforded by $\lambda-1^3$, $\lambda-1^22$,
$\lambda-12^2$, and $\lambda-123$; in addition, for $G=B_3$, $r-6$ is afforded by $\lambda-23^2$ and if $G=C_3$, by $\lambda-2^23$. Hence $n_3\geq 5$ and $V\downarrow A$ is not MF by Lemma~\ref{multiplicity bound general lemma}. 

Case 2 : $\lambda = \omega_1+c\omega_2$.
As in the previous case, we reduce to $c=p-2$, so $c\geq 5$. The $T_A$-weight $r-6$ is afforded by
$\lambda-12^2$, $\lambda-1^22$, $\lambda-2^23$, $\lambda-123$, and $\lambda-2^3$. Therefore $n_3\geq 5$ and $V\downarrow A$ is not MF by Lemma~\ref{multiplicity bound general lemma}. 

Case 3 : $G = B_3$ and $\lambda = c\omega_2+\omega_3$. The $T_A$-weight $r-8$ is
afforded by $\lambda-12^23$, $\lambda-1^22^2$,  $\lambda - 2^33$, $\lambda-23^3$, $\lambda - 123^2$, $\lambda-2^23^2$ which implies $n_4\geq 6$ and $V\downarrow A$ is not MF by Lemma~\ref{multiplicity bound general lemma}.

Case 4 : $G = C_3$ and $\lambda = c\omega_2+\omega_3$.
By Lemma~\ref{1 dimensional weight space lemma} we may assume that $c+4\equiv 0\mod p$, implying $c\geq 3$. The $T_A$-weight $r-6$ is afforded by $\lambda-2^3$, $\lambda-12^2$, $\lambda-2^23$, $\lambda-123$,
$\lambda-23^2$, which implies $n_3\geq 5$ and $V\downarrow A$ is not MF by Lemma~\ref{multiplicity bound general lemma}.

Case 5 :  $G = B_3$ and $\lambda = \omega_2+c\omega_3$. As above we reduce to $c+4\equiv 0\mod p$ and so $c\geq 3$. The $T_A$-weight $r-6$ is afforded by $\lambda-123$, $\lambda -2^23$, $\lambda-3^3$ and $\lambda-23^2$. By Lemma~\ref{B2 Weyl}, the Weyl module $\Delta_{B_2}(1c)$ has exactly two composition factors $L_{B_2}(1c)$ and $L_{B_2}(0c)$, the latter afforded by $\lambda-11$. Therefore the multiplicity of the weight $\lambda-23^2$ in $V$ is $2$ and so $n_3\geq 5$, which by Lemma~\ref{multiplicity bound general lemma}
implies that $V\downarrow A$ is not MF. 

Case 6 :  $G = C_3$ and $\lambda = \omega_2+c\omega_3$.
This is entirely similar. Here we may assume $2c+3\equiv 0\mod p$. If $c\geq 4$, the $T_A$-weight 
$r-8$ is afforded by $\lambda-12^23$, $\lambda - 2^33$, $\lambda - 23^3$, $\lambda - 2^23^2$, $\lambda - 123^2$, and $\lambda - 3^4$. Therefore $n_4\geq 6$, and $V\downarrow A$ is not MF by Lemma~\ref{multiplicity bound general lemma}.
If $c\leq 3$, we must have $c =2$ and $p=7$. By \cite{lubeckOnline}, all weight spaces of $V$ are $1$-dimensional.
The $T_A$-weight $r-8$ is again afforded by the first $5$ weights listed above. On the other hand, the $T_A$-weight $r-6$ is afforded precisely by $ \lambda-123,\lambda-23^2$ and $\lambda-2^23$, implying that $n_3 = 3$. Therefore $n_4-n_3\geq 2$, and by Lemma~\ref{sufficient conditions MF lemma} we conclude that $V\downarrow A$ is not MF. 
\end{proof}

\begin{lemma}\label{A3 B3 C3 11 lemma}
    Let $G=A_3,\,B_3$ or $C_3$ and $\lambda = \omega_1+\omega_2$ or $\omega_2+\omega_3$.
    Then $V\downarrow A$ is not MF.
\end{lemma}

\begin{proof}
Consider $G=A_3$. Then $r=7$ and $p=5$ or $p=7$ as $r\geq p\geq h$. In both cases the Weyl module is irreducible and the conditions for Lemma~\ref{CD lemma} are satisfied, implying that $V\downarrow A$ is not MF. 

    Now consider $G=B_3$. Since $p\geq 7$, the Weyl module is irreducible. The conditions for Lemma~\ref{CD lemma} are satisfied, implying that $V\downarrow A$ is not MF. 

    Finally consider $G=C_3$. If $(\lambda,p)\neq (\omega_1+\omega_2,7)$ we can conclude as for $B_3$. Therefore assume that $\lambda = \omega_1+\omega_2$ with $p=7$. By Lemma~\ref{1 dimensional weight space lemma} we have $\dim V_{\lambda-123}=2$, and it is straightforward to see that $n_3\geq 5$, which by Lemma~\ref{multiplicity bound general lemma} implies that $V\downarrow A$ is not MF.
\end{proof}

\begin{lemma}\label{rank at least 3 end nodes lemma}
Assume that  $G$ has rank at least $3$ and $\lambda = \omega_i+\omega_j$ where $\alpha_i$ and $\alpha_j$ are end-nodes. Then $V\downarrow A$ is not MF.
\end{lemma}

\begin{proof}
Consider first the case where $G=A_\ell$, where $p\geq h = \ell+1$. By \cite{lubeck} the Weyl module $\Delta(\lambda)$ is irreducible if and only if $p\neq \ell+1$. If $p\neq \ell+1$, the conditions of Lemma~\ref{CD lemma} are satisfied and therefore $V\downarrow A$ is not MF. We therefore reduce to the case $p = \ell+1$, where $V$ is isomorphic to the quotient of $\Delta(\lambda)$ by a $1$-dimensional trivial submodule. For $d<\ell$, it is straightforward to see that $n_d = d+1$ (where we use that $r=2\ell=2(p-1)$).
Therefore by Lemma~\ref{sufficient conditions MF lemma}{(i)} we know that $(p+1)$ is a composition factor of $V\downarrow A$. Now the $T_A$-weight $p-3$ occurs with multiplicity one more than the $T_A$-weight $p-1$, and it does not occur as a weight in $(p+1)$, while $p-1$ does. Therefore Lemma~\ref{sufficient conditions MF lemma}{(iv)} implies that $V\downarrow A$ is not MF. 

    If $G=B_\ell$ or $C_\ell$ and $\ell\geq 4$, the first paragraph of the proof of \cite[Lemma~3.5]{LSTA1} shows that $n_3\geq 5$, so $V\downarrow A$ is not MF by Lemma~\ref{multiplicity bound general lemma}. If $G=C_3$, we can apply Lemma~\ref{CD lemma} to conclude that $V\downarrow A$ is not MF.
    
Now assume $G=B_3$. We have $p=7$ or $p=11$, as $r=12$. If $p=11$, the Weyl module is irreducible by \cite{lubeck}, and the conditions of Lemma~\ref{CD lemma} are satisfied, implying that $V\downarrow A$ is not MF.
When $p=7$, using \cite{lubeckOnline}, we find that $n_2 = 3$ and $(r-4)$ is therefore a composition factor by Lemma~\ref{sufficient conditions MF lemma}{(i)}. Furthermore, we have $n_3 = 3$, $n_4 = 4$ and the $T_A$-weight $r-8$ does not occur as a weight in $(r-4)$, while the $T_A$-weight $r-6$ does. Therefore  Lemma~\ref{sufficient conditions MF lemma}{(iv)} implies that $V\downarrow A$ is not MF.

    Now consider $G=D_\ell$, with $\ell\geq 4$. If $\lambda = \omega_1+\omega_{\ell-1}$, the $T_A$-weight $r-2(\ell-1)$ is afforded by $\lambda-1\dots (\ell-1)$, $\lambda-2\dots \ell$, and $\lambda-1\dots (\ell-2)\ell$. Since $p\geq h$ we have $p>\ell$, and therefore by Lemma~\ref{lambda minus string dimension lemma} and Lemma~\ref{subdiagram lemma} we have $\dim V_{\lambda-1\dots (\ell-1)} = \ell -1$. Therefore $n_{\ell-1} \geq \ell+1$ and Lemma~\ref{multiplicity bound general lemma} implies that $V\downarrow A$ is not MF. It is also easy to see that if $\lambda = \omega_{\ell-1}+\omega_\ell$, we have $n_3\geq 5$. Again Lemma~\ref{multiplicity bound general lemma} implies that $V\downarrow A$ is not MF.

    Finally, if $G$ is exceptional, the arguments used in the proof of \cite[Lemma~3.6]{LSTA1} in characteristic zero allow us to conclude, as \cite{lubeckOnline} shows that the relevant weight spaces in $V$ have the same dimension as the corresponding weight spaces in $\Delta_K(\lambda)$.
\end{proof}

\begin{proposition}\label{proposition rank at least 3 double support}
    Suppose that $G$ has rank at least $3$ and $\lambda = c_i\omega_i+c_j\omega_j$ with $c_i,c_j\geq 1$. Then $V\downarrow A$ is not MF.
\end{proposition}

\begin{proof}
By Lemma~\ref{rank at least 4 adjacent reduction}, if $\alpha_i$ and $\alpha_j$ are adjacent and $G$ has rank at least $4$, the module $V\downarrow A$ is not MF. If $\alpha_i$ and $\alpha_j$ are adjacent and $G$ has rank $3$, Lemma~\ref{A3 c10 and 1c0 reduction}, Lemma~\ref{B3 and C3 c1 reduction} and Lemma~\ref{A3 B3 C3 11 lemma} combine to imply that $V\downarrow A$ is not MF.

We have therefore reduced to $\alpha_i$ and $\alpha_j$ not adjacent, in which case by Proposition~\ref{first reduction lemma} we can assume that $c_i=c_j=1$ and $\alpha_i$ and $\alpha_j$ are both end-nodes. In this case, by Lemma~\ref{rank at least 3 end nodes lemma} we conclude that $V\downarrow A$ is not MF.
\end{proof}

\subsection{The case where $\lambda = b\omega_i$}

We now consider the case $\lambda = b\omega_i$. Note that if $G$ is classical, then $\lambda\ne \omega_1$, as we are assuming that $p\leq r$, and necessarily $p\geq h$.
\begin{lemma}\label{rank >= 3 2w_1 lemma}
     Assume that $G=A_{\ell}, B_{\ell},C_{\ell}$ with $\ell\geq 3$ or $G=D_{\ell}$ with $\ell\geq 4$. Let $\lambda = b\omega_1$, with $b\geq 2$. Then $V\downarrow A$ is not MF.
\end{lemma}

\begin{proof}
    We first consider the case $b=2$ and start by assuming that $(G,p)\neq (B_\ell,2\ell+1)$. By \cite{lubeck} and since $p\geq h$, the Weyl module is irreducible. A simple check shows that the conditions of Lemma~\ref{CD lemma} are satisfied, implying that $V\downarrow A$ is not MF. Consider now the case $G = B_\ell$ and $p = 2\ell+1$, where $V$ is isomorphic to the quotient of  $\Delta(\lambda)$ by a $1$-dimensional trivial submodule. For all strictly positive weights \(r-2d\), we have $n_d = \dim(\Delta_K(\lambda)\downarrow A_{K})_{r-2d}$. By \cite[Lemma~4.2]{LSTA1}, we have $\Delta_K(\lambda)\downarrow A_K = (4\ell) + (4\ell-4)+\dots $, which implies $n_d = d+1$ for $d$ even with $d< 2\ell $, and $n_{d+1} = n_d$ for $d$ odd with $d+1< 2\ell $. By Lemma~\ref{sufficient conditions MF lemma}{(i)}, for all $0\leq d < 2\ell$ we have that $(r-2d)$ is a composition factor of $V\downarrow A$. In particular, either $(p+1)$ or $(p+3)$ is a composition factor of $V\downarrow A$. In the first case, the $T_A$-weight $p-3$ occurs with multiplicity one more than the $T_A$-weight $p-1$, but does not occur as a weight in $(p+1)$. Therefore Lemma~\ref{sufficient conditions MF lemma}{(iv)} implies that $V\downarrow A$ is not MF. Similarly, if $(p+3)$ is a composition factor of $V\downarrow A$, then the $T_A$-weight $p-5$ (note that $p>5$ since $\ell\geq 3$) occurs with multiplicity one more than the $T_A$-weight $p-3$, but does not occur as a weight in $(p+3)$, concluding in the same way.

    Now consider the case $b\geq 3$. Start with $G=A_\ell$. Here $V = \Delta(\lambda)$ by \textup{\cite[\nopp 1.14]{seitz_mem_class}}. 
    If $\ell \geq 6$, the first paragraph of the proof of \cite[Lemma~4.4]{LSTA1} shows that $n_6\geq 7$, which by Lemma~\ref{sufficient conditions MF lemma}{(iii)} implies that $V\downarrow A$ is not MF. 
    If $b = 4$ and $\ell = 4$ or $\ell = 5$, we similarly have $n_4\geq 5$. If $b\geq 5$ and $(b,\ell)\neq (5,3)$, we have $n_6-n_5\geq 2$. Therefore by Lemma~\ref{sufficient conditions MF lemma}{(iii)} and {(ii)}, we reduce down to the cases $(b,\ell) = (4,3)$, $(5,3)$, $(3,3)$, $(3,4)$, $(3,5)$. For these cases we can conclude using Lemma~\ref{CD lemma}, unless $\ell=b=3$ and $p=5$. In this case $r = 9$ and the weights $9,7,5,3$ occur respectively with multiplicities $1,1,2,3$. By Lemma~\ref{sufficient conditions MF lemma}{(i)}, we have that $(5)$ is a composition factor, and in addition $r-6$ does not occur as weight in this composition factor. Therefore by  Lemma~\ref{sufficient conditions MF lemma}{(iv)} the module $V\downarrow A$ is not MF.
     
     The $C_\ell$ case follows from the $A_{2\ell-1}$ case since $A<C_{\ell}<A_{2\ell-1}$ is a principal $A_1$-subgroup of $A_{2\ell-1}$ and $V=S^b(L_{C_{\ell}}(\omega_1)) = L_{A_{2\ell-1}}(b\omega_1)\downarrow C_\ell$. Now consider $G=B_\ell$. If $b\geq 4$, it is straightforward to check that $n_4\geq 5$, which by Lemma~\ref{sufficient conditions MF lemma}{(iii)} implies that $V\downarrow A$ is not MF. If $b=3$ and $\ell\geq 4$, by the proof of \cite[Lemma~4.4]{LSTA1}, we have $n_6\geq 7$, concluding in the same way. If $\ell=b=3$ we have $p\leq r=18$ and by Lemma~\ref{CD lemma} the module $V\downarrow A$ is not MF. Finally, we consider the case where $G=D_\ell$ where we have $A\leq B_{\ell-1}<G$. Since $\Delta_{B_{\ell-1}}(b\omega_1)$ is a composition factor of $\Delta_{D_{\ell}}(b\omega_1)$, if $\Delta_{D_{\ell}}(b\omega_1)\downarrow A$ is MF, so is $\Delta_{B_{\ell-1}}(b\omega_1)$. Therefore by the $B_\ell$ result, we conclude that $V\downarrow A$ is not MF.
 \end{proof}

\begin{lemma}\label{b omega_i i>1 b>1 lemma}
Assume that $G=B_\ell,C_\ell$ with $\ell\geq 3$ or $G=D_\ell$ with $\ell\geq 4$. Let $\lambda = b\omega_i$, with $i>1$ and $b>1$. Then $V\downarrow A$ is not MF.
 \end{lemma}

 \begin{proof}
 By Lemma~\ref{bw_i reduction lemma} we can assume that $\lambda = b\omega_\ell$. We will treat the case $G = D_\ell$ at the end of the proof.
 
Assume for now that $b\geq 3$. If $G=C_\ell$, the $T_A$-weight $r-6$ is afforded by $\lambda-\ell^3$, $\lambda-(\ell-2)(\ell-1)\ell$, $\lambda-(\ell-1)^2\ell$,
   $\lambda-(\ell-1)\ell^2$. If $G=B_\ell$, the $T_A$-weight $r-6$ is afforded by $\lambda-(\ell-1)\ell^2$, $\lambda-(\ell-2)(\ell-1)\ell$ and $\lambda-\ell^3$, and using the fact
   that the $B_2$-module $\Delta(b\omega_2)$ is irreducible, by Lemma~\ref{subdiagram lemma}, we have that the first of these weights has multiplicity $2$. Hence for both of the groups $C_\ell$ and $B_\ell$, we have $n_3\geq 4$. By Lemma~\ref{sufficient conditions MF lemma}{(iii)}, the module $V\downarrow A$ is not MF. 

We now consider the case $b=2$ when $G = C_\ell$ and first assume that $\ell\geq 5$. The $T_A$-weight $r-10$ is afforded by $\lambda-(\ell-1)^3\ell^2$, $\lambda-(\ell-2)(\ell-1)^2\ell^2$, $\lambda-(\ell-1)^2\ell^3$, $\lambda-(\ell-2)^2(\ell-1)^2\ell$,
   $\lambda-(\ell-3)(\ell-2)(\ell-1)\ell^2$ and $\lambda-(\ell-4)(\ell-3)(\ell-2)(\ell-1)\ell$. Again, by Lemma~\ref{sufficient conditions MF lemma}{(iii)} the module $V\downarrow A$ is not MF.

   Now consider the cases $C_\ell$, for $\lambda = 2\omega_\ell$ and $\ell=3,4$ where $r=18$, respectively $32$, and $p\geq 7$, respectively $11$. In both cases we have that $\Delta(2\omega_\ell)$ is irreducible by \cite{lubeck}. For $\ell=3$, the conditions of Lemma~\ref{CD lemma} are satisfied, implying that $V\downarrow A$ is not MF.
   If $\ell=4$, by the first paragraph of \cite[Lemma~4.3]{LSTA1} we have $\dim (\Delta_K(\lambda)\downarrow A_{K})_{r-12} \geq \dim (\Delta_K(\lambda)\downarrow A_{K})_{r-10}+2$. Therefore by Lemma~\ref{n_d same as p=0 when Weyl module irreducible lemma} we find that $n_6-n_5\geq 2$, concluding by Lemma~\ref{sufficient conditions MF lemma}{(ii)}.  
 
   Turn now to the case $G = B_\ell$ and $b=2$. Here the $T_A$-weight $r-8$ is afforded by
   $\lambda-(\ell-2)(\ell-1)\ell^2$, $\lambda-(\ell-1)^2\ell^2$,
   $\lambda-(\ell-1)\ell^3$ and if $\ell\geq4$, $\lambda-(\ell-3)(\ell-2)(\ell-1)\ell$. The first of these weights is conjugate to $\lambda-(\ell-1)\ell^2$ and so
   has multiplicity $2$ by the first paragraph of this proof. Therefore
   if $\ell\geq 4$, $V\downarrow A$ is not MF by Lemma~\ref{sufficient conditions MF lemma}{(iii)}. So finally, we reduce to $b=2$ and $\ell=3$, where $p\geq 7$ and $r = 12$. The Weyl module is irreducible by \cite{lubeck}. The conditions of  Lemma~\ref{CD lemma} are satisfied, implying that $V\downarrow A$ is not MF.

   Finally suppose that $G = D_\ell$ and $\lambda = b\omega_\ell$. Since $A\leq  B_{\ell-1} \leq D_{\ell}$ and $V\downarrow B_{\ell-1}\cong L_{B_{\ell-1}}(b\omega_{\ell-1})$, we
   may use the $B_{\ell-1}$ result to
   conclude.
 \end{proof}

 \begin{lemma}\label{En case cw_i c>1 lemma}
 If $G = E_\ell$ and $\lambda=b\omega_i$ with $b>1$, then  $V\downarrow A$ is not MF.  
 \end{lemma}

 \begin{proof}
     This follows verbatim from the proof of \cite[Lemma~4.6]{LSTA1}, unless $i=\ell$ and $G=E_7$ or $E_8$ with $b=2$ or $b=3$. In these remaining cases, it is not difficult to check that we have $n_6\geq n_5+2$ (as stated in the proof of \cite[Lemma~4.6]{LSTA1}), as this count relies on $1$-dimensional weight spaces. By Proposition~\ref{first reduction lemma} the module $V\downarrow A$ is not MF.
 \end{proof}

  \begin{lemma}\label{F4 case cw_i c>1 lemma}
 If $G = F_4$ and $\lambda=b\omega_i$ with $b>1$, then $V\downarrow A$ is not MF.  
 \end{lemma}

 \begin{proof}
    By Lemma~\ref{bw_i reduction lemma} the simple root $\alpha_i$ corresponds to an end-node of the Dynkin diagram.
     If $i=1$ we can conclude as in the first paragraph of the proof of \cite[Lemma~4.7]{LSTA1}.

     Assume $i=4$. If $b\geq 3$, like in \cite[Lemma~4.7]{LSTA1} we have $n_4\geq 5$, concluding by Lemma~\ref{sufficient conditions MF lemma}{(iii)}. If $b=2$ we have $V = \Delta(\lambda)$ by \cite{lubeck}, and since $r=32$ and $r\geq p>11$, the conditions of Lemma~\ref{CD lemma} are satisfied. Thus, $V\downarrow A$ is not MF.
 \end{proof}

It remains to consider the case $\lambda = \omega_i$. Recall that for $G$ classical, we have $\lambda\ne\omega_1$.

 \begin{lemma}\label{omega_i lots of cases classical lemma}
     Assume that  $G$ has rank at least $3$, $\lambda = \omega_i$ and that one of the following holds.
     \begin{enumerate}[label=\textnormal{(\roman*)}]
        \item $G=A_\ell,B_\ell,C_\ell$ with $\ell\geq 3$  or $G=D_{\ell}$ with $\ell\geq 4$, and $4\leq i\leq \ell-3$.
        \item $G=A_\ell$, $i=3$, and $\ell\geq 5$.
        \item $G=A_\ell,B_\ell, C_\ell$ with $\ell\geq 3$  or $G=D_{\ell}$ with $\ell\geq 4$ and  $i=2$.
     \end{enumerate}
     Then $V\downarrow A$ is not MF. 
 \end{lemma}
 \begin{proof}
 
    \begin{enumerate}[label=\textnormal{(\roman*)}]
         \item Lemma~\ref{w_i reduction space on both sides} applies, except when $G = D_7$, and implies that the module $V\downarrow A$ is not MF. For the case $G = D_7$, where $i=4$, it is straightforward to see that $n_4\geq 5$ and then Lemma~\ref{sufficient conditions MF lemma}(iii) implies that $V\downarrow A$ is not MF.
         \item Here $V= \bigwedge^3(L(\omega_1))$. Assume for now that $l\geq 8$. The $T_A$-weight $r-12$ is afforded by the wedge of weight vectors in $L(\omega_1)$ for each of the following triples of $T_A$-weights: $\ell(\ell-2)(\ell-16)$, $\ell(\ell-4)(\ell-14)$, $\ell(\ell-6)(\ell-12)$, $\ell(\ell-8)(\ell-10)$, $(\ell-2)(\ell-4)(\ell-12)$, $(\ell-2)(\ell-6)(\ell-10)$, $(\ell-4)(\ell-6)(\ell-8)$. Therefore $n_6\geq 7$, and Lemma~\ref{sufficient conditions MF lemma}{(iii)} implies that $V\downarrow A$ is not MF.

         For the remaining cases, when $5\leq l\leq 7$, we have $\Delta(\lambda) = V$ and a quick check shows that the conditions of Lemma~\ref{CD lemma} are satisfied, implying that $V\downarrow A$ is not MF.
         \item Here $\lambda = \omega_2$, and as $p>\ell$ we have $V = \Delta(\lambda)$ \cite[Table~2]{lubeck}. We have  $r = 2\ell-2,4\ell-2,4\ell-4$ or $r=4\ell-6$ according to whether $G=A_\ell,B_\ell,C_\ell$ or $G=D_\ell$. Furthermore we have $p$ greater than $\ell,2\ell-1,2\ell,2\ell-2$ respectively. It is then an easy check to see that the conditions of Lemma~\ref{CD lemma} are satisfied, implying that $V\downarrow A$ is not MF.
     \end{enumerate}\end{proof}

     \begin{lemma}\label{omega_i i>=3 not spin lemma}
         Assume that $G=B_\ell,C_\ell$ with $\ell\geq 3$  or $G=D_{\ell}$ with $\ell\geq 4$, and  in addition that $\lambda = \omega_i$ for $i\geq 3$ and $V$ is not a spin module for $B_\ell$ or $D_\ell$. Then $V\downarrow A$ is not MF.
     \end{lemma}

     \begin{proof}
         If $G=B_\ell$ or $D_\ell$, then $V = \bigwedge^i(\omega_1)$ by \cite{seitz_mem_class} and the result follows from Lemma~\ref{omega_i lots of cases classical lemma}{(i)(ii)} for $G=A_{2\ell}$ or $A_{2\ell-1}$. Indeed, if $G = B_\ell$, then $A$ is regular in $A_{2\ell}$ and $V = L_{A_{2\ell}}(\omega_i)\downarrow G$, while if $G = D_\ell$, then $A < B_{\ell-1} < D_{\ell}$ and by the $B_\ell$ case there is a $B_{\ell-1}$-composition factor of $V$ (namely $L_{B_{\ell-1}}(\omega_i)$) that is not multiplicity-free in its restriction to $A$, implying that $V\downarrow A$ is not MF. 
         
         We now consider the case $G=C_\ell$. By part (i) of Lemma~\ref{omega_i lots of cases classical lemma} we can furthermore assume that $i = 3$ or $i>\ell-3$. If $i = \ell-2>3$, the $T_A$-weight $r-8$ has multiplicity at least $5$ as it is afforded by $5$ different weights as in the proof of \cite[Lemma~5.3]{LSTA1}. Therefore Lemma~\ref{sufficient conditions MF lemma}{(iii)} implies that $V\downarrow A$ is not MF.

         Assume $i=\ell-1>3$. As $\ell\geq 5$, the $T_A$-weight $r-12$ is afforded by $\lambda-(\ell-3)(\ell-2)(\ell-1)^3\ell$, $\lambda-(\ell-2)(\ell-1)^3\ell^2$, $\lambda-(\ell-2)^2(\ell-1)^3\ell$, $\lambda-(\ell-4)(\ell-3)(\ell-2)(\ell-1)^2\ell$, $\lambda-(\ell-3)(\ell-2)^2(\ell-1)^2\ell$. When $\ell = 5$, the last two weights have multiplicity $2$ in $V$ by \cite{lubeckOnline}, and therefore the same holds for $\ell\geq 5$ by Lemma~\ref{subdiagram lemma}. Thus, $n_6\geq 7$, and by Lemma~\ref{sufficient conditions MF lemma}{(iii)} the module $V\downarrow A$ is not MF.

         Now assume $i=\ell>3$. Start with $\ell =4$ or $5$. In both cases the Weyl module is irreducible. We have $r=16$ if $\ell = 4$, and $r=25$ if $\ell = 5$. If $(\ell,p)\ne (5,13)$, the conditions of Lemma~\ref{CD lemma} are satisfied, showing that $V\downarrow A$ is not MF. In the remaining case (when $(\ell,p) = (5,13)$), we find that $B(r) - \dim (r-2) = 118<\dim V = 132$. Therefore by Lemma~\ref{dimension bound for omega_i}, the module $V\downarrow A$ is not MF. 
         
         Now consider the case $\ell\geq 6$ with $\lambda = \omega_\ell$. Here the $T_A$-weight $r-10$ has multiplicity $4$ as it is afforded precisely by $\lambda- (\ell-4)(\ell-3)(\ell-2)(\ell-1)\ell$, $\lambda- (\ell-2)^2(\ell-1)^2\ell$, $\lambda- (\ell-3)(\ell-2)(\ell-1)^2\ell$, $\lambda- (\ell-2)(\ell-1)^2\ell^2$. The $T_A$-weight $r-12$ has multiplicity at least $6$ as it is afforded by $\lambda- (\ell-5)(\ell-4)(\ell-3)(\ell-2)(\ell-1)\ell$, $\lambda- (\ell-2)^2(\ell-1)^2\ell^2$, $\lambda- (\ell-4)(\ell-3)(\ell-2)(\ell-1)^2\ell$, $\lambda- (\ell-3)(\ell-2)^2(\ell-1)^2\ell$, $\lambda- (\ell-3)(\ell-2)(\ell-1)^2\ell^2$, $\lambda- (\ell-2)(\ell-1)^3\ell^2$. By Lemma~\ref{sufficient conditions MF lemma}{(ii)} $V\downarrow A$ is not MF.

         We now turn to the final case, where $i=3$. If $\ell\geq 6$, we have $n_6\geq 7$, since by the last paragraph of the proof of \cite[Lemma~5.3]{LSTA1} there are $7$ distinct weights of $V$ affording the $T_A$-weight $r-12$. Therefore $V\downarrow A$ is not MF by Lemma~\ref{sufficient conditions MF lemma}{(iii)}.

         In the remaining cases, when $\ell\in\{3,4,5\}$, the Weyl module is irreducible and we can apply Lemma~\ref{CD lemma}, unless $\ell=5$ and $p=11$, in which case $r=21\equiv -1\mod p$.
         In this case, we find that $B(r) - \dim (r-2) = 84<\dim V = 110$. Therefore by Lemma~\ref{dimension bound for omega_i}, the module $V\downarrow A$ is not MF.
     \end{proof}

     \begin{lemma}\label{spin modules MF classification lemma}
         Assume that $V$ is a spin module for $B_\ell$ with $\ell\geq 3$ or $D_\ell$ with $\ell\geq 4$. Then $V\downarrow A$ is not MF.
     \end{lemma}

     \begin{proof}
        We have $V = \Delta(\lambda)$.
         If $G=D_\ell$, then $A\leq B_{\ell-1}<G$ and $V\downarrow B_{\ell-1}$ is the spin module for $B_{\ell-1}$; therefore it suffices to prove the lemma for $G=B_\ell$, where $r=\ell(\ell+1)/2$ and $\dim V = 2^\ell$. If $V\downarrow A$ is MF the dimension of $V$ is at most $B_{K}(r)$, by Lemma~\ref{dimension bound comparison}. This implies that if $\ell\geq 10$, the module $V\downarrow A$ is not MF. 
         
         Now assume $\ell\leq 9$. Since $p>h=2\ell$ we know that $p\nmid r$. Therefore if $V\downarrow A$ is MF the dimension of $V$ is at most $B(r)-\dim(r-2)$, by Lemma~\ref{dimension bound for omega_i}. This then reduces our considerations to the pairs $(n,p)$ in the following list:
         $(5,11)$, $(5,13)$, $(6,13)$, $(6,17)$, $(6,19)$, $(7,17)$, $(7,19)$, $(7,23)$, $(8,31)$. For every $3\leq \ell\leq 8$, by Lemma~\ref{n_d same as p=0 when Weyl module irreducible lemma} we can read the dimension of the $T_A$-weight space $r-2k$ off the table in the proof of \cite[Lemma~5.4]{LSTA1}. In each case we apply part (iii) of Lemma~\ref{sufficient conditions MF lemma} to find that $V\downarrow A$ is not MF. The first repeated composition factors are of highest weight respectively $5,9,9,11,15,14,14,16,24$.
     \end{proof}

     \begin{lemma}\label{exceptionals w_i calssification lemma}
        Assume $G=E_\ell$ or $F_4$ and $\lambda=\omega_i$. Then $V\downarrow A$ is not MF.
     \end{lemma}

     \begin{proof}
    If $G=E_\ell$ and $\lambda = \omega_4$, the $T_A$-weight $r-4$ is afforded by $\lambda - 34$, $\lambda-24$, $\lambda - 45$. Therefore $n_2\geq 3$, and by Lemma~\ref{sufficient conditions MF lemma}{(iii)}, the module $V\downarrow A$ is not MF.

    If $G = E_8$ and $\lambda = \omega_3$ or $\omega_6$, then $r = 182 $ respectively $r=168$, giving $B_{K}(r)=8464$ and $7225$ respectively. By \cite{lubeck}, we have $\dim V > B_{K}(r)$ and therefore by Lemma~\ref{dimension bound comparison} the module $V\downarrow A$ is not MF. If $G=E_8$ and $\lambda = \omega_5$, by Lemma~\ref{w_i reduction space on both sides} the module $V\downarrow A$ is not MF.

    In all remaining cases, by \cite{lubeck} we have that $V= \Delta(\lambda)$. Lemma~\ref{dimension bound comparison} then allows us to reduce to the case where $V$ is the minimal module for $G$, or the adjoint module for $E_6,E_7$ or $F_4$. The conditions of Lemma~\ref{CD lemma} are satisfied, implying that $V\downarrow A$ is not MF.
     \end{proof}

     \begin{proposition}\label{rank>=3 single support main proposition}
        Suppose that $G$ has rank at least $3$ and $\lambda = b\omega_i$, with $b\geq 2$ for $G$ classical and $b\geq 1$ for $G$ of exceptional type. Then $V\downarrow A$ is not MF.
\end{proposition}

\begin{proof}
    If $G$ is classical, this is a direct consequence of Lemma~\ref{rank >= 3 2w_1 lemma} and Lemma~\ref{b omega_i i>1 b>1 lemma}. 
    If $G$ is exceptional and $b\geq 2$, we similarly conclude by Lemma~\ref{En case cw_i c>1 lemma} and Lemma~\ref{F4 case cw_i c>1 lemma}.

    If $b=1$, where $G$ is exceptional, we reach the same conclusion by Lemmas \ref{omega_i lots of cases classical lemma}, \ref{omega_i i>=3 not spin lemma}, \ref{spin modules MF classification lemma}, \ref{exceptionals w_i calssification lemma}.
\end{proof}

\renewcommand*{\proofname}{Proof of Proposition~$\ref{main prop rank >=3}$.}
\begin{proof}
    If $V$ is MF, then by Proposition~\ref{first reduction lemma}, the weight $\lambda$ is of the form $c_i\omega_i+c_j\omega_j$. If $c_ic_j\neq 0$, then the conclusion follows from Proposition~\ref{proposition rank at least 3 double support}. If $c_ic_j = 0$, then the conclusion follows from Proposition~\ref{rank>=3 single support main proposition}.
\end{proof}
\renewcommand*{\proofname}{Proof.}

\section{Proof of Corollary~\ref{main corollary}}

In this Section we prove Corollary~\ref{main corollary}, thereby extending Theorem~\ref{main theorem} to the case where $\lambda$ is not $p$-restricted. The following lemma serves as an inductive tool.
\begin{lemma}\label{inductive lemma for tensor product}
    Let $\lambda = \sum_{i=0}^t p^i\lambda_i$ where $\lambda_i$ is a $p$-restricted dominant weight for all $0\leq i \leq t$. Assume that for some $s$ with $0\leq s<t$, we have $\left(\sum_{i=0}^s p^i\lambda_i\right)\downarrow T_A < p^{s+1}$. Then $V\downarrow A$ is MF if and only if the following conditions hold.
    \begin{enumerate}[label=\textnormal{(\roman*)}]
        \item $L(\sum_{i=0}^s p^i\lambda_i)\downarrow A$ is MF, and
        \item $L(\sum_{i=s+1}^t p^i\lambda_i)\downarrow A$ is MF.
    \end{enumerate}
\end{lemma}

\begin{proof}
    Let $V_1 = L(\sum_{i=0}^s p^i\lambda_i)$ and $V_2 = L(\sum_{i=s+1}^{t} p^{i}\lambda_i)$, so that $V = V_1\otimes V_2$. If $V_2 = L(0)$, the statement is trivial. Thus, assume $V_2\neq L(0)$.

    One direction is clear. If either $V_1\downarrow A$ or $V_2\downarrow A$ is not MF, then $V\downarrow A$ is not MF. Assume now that both $V_1\downarrow A$ and $V_2\downarrow A$ are MF, and let $V_1\downarrow A$ have composition factors \((r_0),(r_1),\dots , (r_m)\), so that by the assumption on $s$, we have $p^{s+1}>r_0>r_1>\dots >r_m$. 
    Similarly let $V_2\downarrow A$ have composition factors $(v_0),(v_1),\dots , (v_n)$ where $v_0>v_1>\dots >v_n \geq p^{s+1}$. Then for all $0\leq i \leq m$ and $0\leq j \leq n$ we have $(r_i)\otimes (v_j) \cong (r_i+v_j)$, since $r_i<p^{s+1}$ and $v_j\geq p^{s+1}$. This implies that the composition factors of $V_1\otimes V_2$ are precisely of the form $(r_i+v_j)$, which are all clearly distinct. Therefore $V\downarrow A$ is MF.
\end{proof}

Let us restate, and prove, Corollary~\ref{main corollary}.
\begin{corollary}
    Let $\lambda = \sum _{i=0}^t p^i \lambda_i$ where each $\lambda_i$ is a $p$-restricted dominant weight and set $r_i = \lambda_i\downarrow T_A$, for $0\leq i\leq t$. Then $V\downarrow A$ is MF if and only if one of the following holds.
    \begin{enumerate}[label=\textnormal{(\roman*)}]
  \item We have $p> r_i$ and $\Delta_K(\lambda_i)\downarrow A_K$ is MF for all $0\leq i\leq t$.
    \item The group $G$ is of type $A_2$, $p=3$ and there exists $0\leq i\leq t$ such that $\lambda_i = \omega_1+\omega_2$. Moreover, for all $0\leq j\leq t$ we have  $\lambda_j \in \{0,\omega_1+\omega_2,\omega_1,\omega_2\}$ and if $\lambda_j = \omega_1+\omega_2$ for some $0\leq j\leq t-1$, then $\lambda_{j+1} = 0$.
      \item The group $G$ is of type $B_2$, $p=5$ and there exists $0\leq i\leq t$ such that $\lambda_i = 2\omega_1$. Moreover, for all $0\leq j\leq t$ we have  $\lambda_j \in \{0,2\omega_1,\omega_1,\omega_2\}$ and if $\lambda_j = 2\omega_1$ for some $0\leq j\leq t-1$, then $\lambda_{j+1}\in \{0,\omega_2\} $.
     \end{enumerate}
\end{corollary}

\begin{proof}
    We use induction on $t$. If $t=0$, then $\lambda$ is $p$-restricted and the statement follows from Theorem~\ref{main theorem}. Suppose now that $t>0$ and that the statement is valid for all $0\leq t \leq N$ for some $N\in \mathbb{N}$. Let $t = N+1$ and $V_1 = L(\lambda_0)$, $V_2 = L(\sum_{i=1}^{t} p^{i}\lambda_i)$. If $V_1$ or $V_2$, is the trivial $\mathcal kG$-module, then we can conclude by the inductive assumption (since the Frobenius twist of a module $M$ is MF if and only if the module $M$ is MF). Therefore we can assume that $V_1$ and $V_2$ are non-trivial.

    Suppose first that $V\downarrow A$ is MF. Then certainly $V_1\downarrow A$ and $V_2\downarrow A$ are both MF.
    If $r_0< p$, by Lemma~\ref{inductive lemma for tensor product} and the inductive assumption, we conclude that $(G,\lambda,p)$ is as in one of the three conclusions of the statement. Therefore assume that $r_0\geq p$. By Theorem~\ref{main theorem} we have $G=A_2$, $p=3$ and $\lambda_0 = 11$, or $G=B_2$, $p=5$ and $\lambda_0 = 20$.

    Consider first the case $G=A_2$. By Theorem~\ref{main theorem} and Table~\ref{tab:p=0 big table}, we must have $\lambda_i\in \{ 0,11,10,01\}$ for all $0\leq i \leq t$. If $\lambda_1 = 0$, we conclude by the inductive assumption combined with Lemma~\ref{inductive lemma for tensor product} for $s=1$. If $\lambda_1\in \{ \omega_1+\omega_2,\omega_1,\omega_2\}$ then $V_1\downarrow A\otimes L(p\lambda_1)\downarrow A$ has $(4)$ as a repeated composition factor and so $V\downarrow A$ is not MF. For $G=B_2$, by Theorem~\ref{main theorem} and 
    Table~\ref{tab:p=0 big table}, we have $\lambda_1\in\{0,\omega_1,2\omega_1, \omega_2\}$ and a straightforward computation shows that $V_1\downarrow A\otimes L(p\lambda_1)\downarrow A$ has a repeated composition factor for $\lambda_1\in\{\omega_1, 2\omega_1\}$. 

    Suppose now that  (i) holds, then $V\downarrow A$ is MF by the inductive assumption combined with Lemma~\ref{inductive lemma for tensor product}.

    If (ii) or (iii) holds, it is easy to verify that the conditions of Lemma~\ref{inductive lemma for tensor product} with $s = 1$ are satisfied, concluding again by the inductive assumption.\end{proof}

\renewcommand*{\proofname}{Proof.}

\printbibliography

\end{document}